\documentclass[11pt]{amsart}
\usepackage[margin=1in]{geometry}
\usepackage{amsmath, amsfonts, amssymb, amsthm}
\usepackage{amsrefs}
\usepackage{enumitem} 
\usepackage{hyperref}
\usepackage{color}
\usepackage{dsfont}
\usepackage{csquotes}
\usepackage{epigraph}
\usepackage{comment}
\usepackage{mlmodern}
\usepackage[T1]{fontenc}
\usepackage{color}
\definecolor{darkblue}{rgb}{0.0,0.0,0.3}
\hypersetup{
    colorlinks=false,
    linkcolor=blue,
    urlcolor=darkblue,
    }

\newtheorem{theorem}{Theorem}[section]

\newtheorem{lemma}[theorem]{Lemma}

\theoremstyle{definition}

\newtheorem{claim}[theorem]{Claim}
\newtheorem{definition}[theorem]{Definition}

\newtheorem{notation}[theorem]{Notation}

\theoremstyle{remark}
\newtheorem{remark}[theorem]{Remark}

\numberwithin{equation}{section}

\newcommand{\bC}{\mathbb{C}}

\newcommand{\bR}{\mathbb{R}}
\newcommand{\bH}{\mathbb{H}}
\newcommand{\bS}{\mathbb{S}}

\newcommand{\Ampere}{Amp\`{e}re}

\newcommand{\Garding}{G\r{a}rding}
\newcommand\norm[1]{\left\lVert#1\right\rVert}
\NewDocumentCommand{\sff}{}{\mathrm{I\!I}}

\DeclareMathOperator{\Span}{Span}

\title[Prescribed Curvature Equations in Warped Product Manifolds]{Starshaped Compact Hypersurfaces in Warped Product Manifolds I: Prescribed Curvature Equations}
\author{Bin Wang}
\address[]{Department of Mathematics, the Chinese University of Hong Kong, Shatin, N.T., Hong Kong.} 
\email{bwang@math.cuhk.edu.hk}

\thanks{}

\begin{document}
\maketitle
\begin{abstract}
We derive global curvature estimates for closed, strictly star-shaped $(n-2)$-convex hypersurfaces in warped product manifolds, which satisfy the prescribed $(n-2)$-curvature equation with a general right-hand side. The proof can be readily adapted to establish curvature estimates for semi-convex solutions to the general $k$-curvature equation. Furthermore, it can also be used to prove the same estimates for $k$-convex solutions to the prescribed curvature measure type equations.

\end{abstract}

\tableofcontents

\section{Introduction}
It is evident that the study of the $\sigma_k$ equation has become a central objective in global differential geometry. A few classical examples include: the Minkowski problem \cite{Alexandroff, Minkowski, Lewy, Cheng-Yau, Nirenberg, Pogorelov-1952, Pogorelov-1978}, the Alexandrov-Chern problem \cite{Alexandrov, Chern, Chern-ICM}, and the problem of finding closed strictly convex hypersurfaces whose Weingarten curvature is prescribed on $\bS^n$ in terms of its inverse Gauss map \cite{Guan-Guan}. More recent examples include: the prescribed curvature measure problem \cite{Guan-Lin-Ma, Guan-Li-Li, Yang}, the $\sigma_k$-Nirenberg problem \cite{LNW-2024, LNW-2021, LLL}, and the Christoffel-Minkowski problem \cite{CM-1,CM-2,CM-3}.

Here we are interested in a natural variation of the aforementioned problems stated as follows: under what conditions a given sufficiently smooth, positive function $\Phi(X,\nu)$ is the $k$-th mean curvature $\sigma_k(\kappa(X))$ of a smooth closed hypersurface $\Sigma$ embedded in a space form of constant sectional curvature as a graph over a sphere? 

Formulating it as a PDE problem \footnote{See e.g. \cite{BDO} for the PDE formulation.}, we are concerned with the existence of star-shaped compact solutions $\Sigma$ in space forms to the curvature equation 
\begin{equation}
\sigma_k(\kappa(X))=\Psi(X,\nu) \quad \text{for all $X \in \Sigma$}, \label{curvature equation 1}
\end{equation}
for which establishing a priori estimates is a crucial ingredient. The $C^0$ and $C^1$ estimates have been derived under some barrier conditions; see e.g. \cite{BLO-1} or \cite{Spruck-Xiao}. However, the $C^2$ estimates had remained open due to the appearance of a gradient term on the right-hand side and can only be obtained in some special cases. 

For $k=1$, the equation is quasi-linear and the required estimates follow from the classical theory of quasi-linear PDEs; see \cite{Treibergs-Wei, BK1, BK2}. While for $k=n$, the equation is of Monge-\Ampere\ type and the estimates follow from Caffarelli-Nirenberg-Spruck's seminal work \cite{CNS1} on the Monge-\Ampere\ equation; see \cite{Oliker-1, Oliker-2, Delanoe, Li}. 

For $2 \leq k \leq n-1$, if $\Psi$ is independent of $\nu$ i.e. $\Psi(X,\nu)=\psi(X)$, then the problem has been completely settled by Caffarelli-Nirenberg-Spruck \cite{CNS4} in $\bR^{n+1}$, by Jin-Li \cite{Jin-Li} in $\bH^{n+1}$, by Li-Oliker \cite{Li-Oliker} in $\bS^{n+1}$ and by De Andrade-Barbosa-De Lira \cite{DBD} in warped product manifolds. On the other hand, if the right-hand side only depends on $\nu$ i.e. $\Psi(X,\nu)=\varphi(\nu)$, this problem is solved in $\bR^{n+1}$ for convex solutions by Guan-Guan in \cite{Guan-Guan}. Ivochkina \cite{Ivochkina-1, Ivochkina-2} considered the Dirichlet problem for \eqref{curvature equation 1} in bounded domains in $\bR^n$ and obtained the $C^2$ estimates under some extra conditions on the dependence of $\Psi$ on $\nu$.

This longstanding problem of obtaining $C^2$ estimates for equation \eqref{curvature equation 1} when $2 \leq k \leq n-1$ and when the right-hand side depends on $\nu$ has recently been raised again by Guan, Li and Li in \cite[Remark 3.5]{Guan-Li-Li}: suppose that there is an a priori $C^1$ bound of $\Sigma$. Can one conclude there is an apriori $C^2$ bound of $\Sigma$ in terms of the $C^1$ norm of $\Sigma$, $\Psi$, $n$ and $k$?

\begin{remark}
The major motivations for studying this problem may come from two aspects \footnote{This question is also a part of one of Yau's problems \cite[problem 59]{Yau-problems}, which concerns the existence of closed hypersurfaces of prescribed genus.}: (i) when $\Psi=\langle X,\nu \rangle \psi(X)$, it corresponds to the equation of the prescribed curvature measure problem, which is important in the study of convex geometry; see \cite{Guan-Lin-Ma, Guan-Li-Li, Yang}. (ii) once we know how to obtain the global curvature bound, the same idea can be applied to obtain global $C^2$ estimates and Pogorelov type interior estimates for the $k$-Hessian equations $\sigma_k(D^2u)=f(x,u,Du)$; see \cite[Proposition 2.3]{Guan-Ren-Wang}, \cite[Corollary 3]{Ren-Wang-1}, \cite[Corollary 5]{Ren-Wang-2}, and \cite{Li-Ren-Wang, Tu} for demonstrations.
\end{remark}

In 2015, a major breakthrough was made by Guan, Ren and Wang \cite{Guan-Ren-Wang}: they established global $C^2$ estimates for 
\begin{itemize}
\item convex hypersurfaces satisfying $\sigma_k(\kappa(X))=\Psi(X,\nu)$ and 
\item 2-convex hypersurfaces satisfying $\sigma_2(\kappa(X))=\Psi(X,\nu)$.
\end{itemize}
Their method was to employ a new test function 
\begin{equation}\label{Guan-Ren-Wang test function}
Q=\begin{cases}
\frac{1}{2}\log \sum_{l} \kappa_{l}^2 - N \log \langle X,\nu \rangle & \text{for convex hypersurfaces}\\
\log \log \sum_{l} e^{\kappa_l}-(1+\varepsilon) \log \langle X,\nu\rangle + \frac{a}{2}|X|^2 & \text{for $2$-convex hypersurfaces} 
\end{cases}
\end{equation}
which is nonlinear in principal curvatures but with sufficiently good convexity properties; this helped them overcome difficulties caused by allowing $\Psi$ to depend on $\nu$. See also the simpler proofs by Chu \cite{Chu} and Spruck-Xiao \cite{Spruck-Xiao} for both results respectively; Chen-Li-Wang \cite{Chen-Li-Wang} later generalized these two results to warped-product manifolds.

By using a modified version of the above test function
\begin{equation}
Q=\log \log \sum_{l} e^{\kappa_l}-N\log \langle X,\nu \rangle, \label{Ren-Wang test function}
\end{equation} Ren and Wang \cite{Ren-Wang-1, Ren-Wang-2} solved, in a rather powerful manner, the same problem in $\bR^{n+1}$ when $k=n-1$ and $k=n-2$. Our first result is a simpler proof for their $C^2$ estimates and we extend them to warped product manifolds. 
\begin{theorem} \label{theorem 2k>n}
Let $k=n-1, n \geq 3$ or $k=n-2, n \geq 5$. Suppose that $\Sigma$ is a closed strictly star-shaped $k$-convex hypersurface in a warped product manifold $\overline{M}$ satisfying the curvature equation \eqref{curvature equation 1} for some positive function $\Psi(X,\nu) \in C^2(\Gamma)$, where $\Gamma$ is an open neighborhood of the unit normal bundle of $\Sigma$ in $\overline{M} \times \bS^n$. Assume that the warping function of $\overline{M}$ is positive and has positive derivative. Then there exists some constant $C>0$ depending only on $n,k, \norm{\Sigma}_{C^1}, \inf \Psi$, $\norm{\Psi}_{C^2}$ and the curvature $\overline{R}$ of $\overline{M}$ such that
\[\max_{X \in \Sigma} \kappa_{\max}(X) \leq C.\]
\end{theorem}

Several remarks are in order. 

\begin{remark}
First, we shall explicitly point out that the $k=n-1$ part of the theorem in warped product manifolds has already been obtained by Chen-Li-Wang in \cite[Section 6]{Chen-Li-Wang}. Their method was based on the ground-breaking paper of Guan-Ren-Wang \cite{Guan-Ren-Wang} and also on how Ren-Wang solved the $k=n-1$ case in $\bR^{n+1}$ \cite{Ren-Wang-1}.

It is then plausible to think that, since the $k=n-2$ case has also been solved by Ren-Wang in $\bR^{n+1}$ \cite{Ren-Wang-2}, so based on their proof and by imitating the method of Chen-Li-Wang, one can readily derive the $C^2$ estimates for the $\sigma_{n-2}$ equation in warped product manifolds as well. Although this result has not been written and published (as far as we know), this method of proof should have already been known to those authors.
\end{remark}

\begin{remark}
Apart from providing a written proof for the $k=n-2$ case, the main contributions of our theorem \ref{theorem 2k>n} are two-fold: on one hand, we approach the $C^2$ estimate by a different method and our proof is simpler than the ones in \cite{Ren-Wang-1, Ren-Wang-2, Chen-Li-Wang}, due to our synthesis of techniques from several other authors. On the other hand, our method (inherited from Lu's paper \cite{Lu-CVPDE} with minor simplifications) is robust and hence is also applicable in other situations concerning curvature estimates for prescribed curvature equations; see theorem \ref{theorem curvature measure} and theorem \ref{weak convexity} below. We believe that this method could be applied to more equations of similar forms \footnote{In a recent preprint \cite{Tu}, Tu has employed Lu's technique to prove Pogorelov type estimates for semi-convex solutions of $k$-Hessian equations.}.
\end{remark}

\begin{remark}
Note that the results of Guan-Ren-Wang \cite{Guan-Ren-Wang} and Ren-Wang \cite{Ren-Wang-1, Ren-Wang-2} were all obtained in the Euclidean space $\bR^{n+1}$. It is then worthwhile to wonder if the estimates could still hold when the ambient space is e.g. the hyperbolic space $\bH^{n+1}$. Due to the nature of $\bH^{n+1}$ having negative curvature, obtaining the $C^2$ estimate in this case is usually more difficult: there would exist an extra negative term $-h_{11}$ which would not exist in the Euclidean case. 

As demonstrated in the work of Spruck-Xiao \cite{Spruck-Xiao}, this problematic term can be handled by applying the maximum principle to $\log \kappa_{\max}$, instead of the more complicated terms in \eqref{Guan-Ren-Wang test function} and \eqref{Ren-Wang test function}. Our proof follows the exact same spirit: we synthesize the test functions of Spruck-Xiao \cite{Spruck-Xiao} and Ren-Wang \cite{Ren-Wang-1, Ren-Wang-2} i.e. \eqref{Spruck-Xiao test function} and \eqref{Ren-Wang test function}, to get a new one
\begin{equation}
Q=\log \kappa_{\max}-N\log \langle X,\nu\rangle + \alpha \Phi  \label{my test function}
\end{equation}
which utilizes the largest principal curvature and gives more "good" third order terms; and we apply a standard perturbation argument as in \cite{Chu} to resolve the issue that $\kappa_{\max}$ may have multiplicity more than one; then by following Lu's derivation in \cite{Lu-CVPDE,Lu-PAMS}, we can remove the large negative term $-\frac{F^{11}h_{111}^2}{h_{11}^2}$ by invoking Ren-Wang's concavity inequality in \cite{Ren-Wang-1, Ren-Wang-2}; finally, it is the elegant paper \cite{Spruck-Xiao} of Spruck and Xiao that made us realize that our proof can be readily carried over to space forms and later to warped product manifolds after \footnote{This note was initially written for curvature estimates in space forms. Later, Lu pointed out the paper of Chen-Li-Wang and suggested the generalization to warped product manifolds.} reading Chen-Li-Wang's paper \cite{Chen-Li-Wang}.
\end{remark}

\begin{remark}
In \cite[Theorem 4.1]{Chen-Li-Wang}, Chen-Li-Wang used the same test function of Spruck-Xiao i.e. \eqref{Spruck-Xiao test function}, to generalize the $k=2$ case from space forms to warped product manifolds, but they did not seem to realize the potential use of the new test function \eqref{my test function}. Because in that paper, Chen-Li-Wang used
\[Q=\begin{cases}
\frac{1}{2}\log \sum_{l} \kappa_{l}^2 - N \log \langle X,\nu \rangle +\alpha \Phi& \text{for convex hypersurfaces}\\
\log \log \sum_{l} e^{\kappa_l}-N \log \langle X,\nu\rangle + \alpha \Phi & \text{for $(n-1)$-convex hypersurfaces} 
\end{cases}\] which is in the same form of Guan-Ren-Wang's test function \eqref{Guan-Ren-Wang test function}. The reader should note that applying the maximum principle to the largest principal curvature is the key to making our proof simpler than theirs \footnote{This has also been observed by Chu \cite{Chu}, and this is how he provided a simple proof for Guan-Ren-Wang's result.}.
\end{remark}

Since the $k=n-1$ case in warped product manifolds is contained in Chen-Li-Wang's work \cite{Chen-Li-Wang}, we shall re-state the theorem for the $k=n-2$ case alone for emphasis.

\begin{theorem}\label{theorem n-2}
Let $n \geq 3$. Suppose that $\Sigma$ is a closed strictly star-shaped ($n-2$)-convex hypersurface in a warped product manifold $\overline{M}$ satisfying the curvature equation 
\[\sigma_{n-2}(\kappa(X))=\Psi(X,\nu) \quad \text{for all $X \in \Sigma$}\] for some positive function $\Psi(X,\nu) \in C^2(\Gamma)$, where $\Gamma$ is an open neighborhood of the unit normal bundle of $\Sigma$ in $\overline{M} \times \bS^n$. Assume that the warping function of $\overline{M}$ is positive and has positive derivative. Then there exists some constant $C>0$ depending only on $n, \norm{M}_{C^1}, \inf \Psi$, $\norm{\Psi}_{C^2}$ and the curvature $\overline{R}$ of $\overline{M}$ such that
\[\max_{X \in \Sigma} \kappa_{\max}(X) \leq C.\]
\end{theorem}

\begin{remark}
When $n=3$, it reduces to the prescribed mean curvature equation which has been solved long ago; when $n=4$, it becomes the $\sigma_2$ equation in dimension four which has been solved by Guan-Ren-Wang \cite[Theorem 1.4]{Guan-Ren-Wang} in $\bR^{n+1}$, by Spruck-Xiao \cite[Theorem 2.1]{Spruck-Xiao} in space forms and by Chen-Li-Wang \cite[Theorem 4.1]{Chen-Li-Wang} in warped product manifolds.

Therefore, here we only need to prove theorem \ref{theorem n-2} for $n \geq 5$; this is also the applicable range of Ren-Wang's inequality, see lemma \ref{Ren-Wang} below. Since the lower order estimates and the existence theorem (once $C^2$ estimates are obtained) are already established in \cite[Theorem 3.3]{Spruck-Xiao} and \cite[Theorem 1.1]{Chen-Li-Wang} (under some barrier conditions), we do not list them here again.
\end{remark}

Next, we state our second result: by using the test function of Spruck-Xiao in \cite{Spruck-Xiao}
\begin{equation}
Q=\log \kappa_{\max} - \log(\langle X,\nu\rangle - a)+\alpha \Phi, \label{Spruck-Xiao test function}
\end{equation} we can prove $C^2$ estimates for $k$-convex solutions to the prescribed curvature measure type equation i.e.
\begin{equation}
\sigma_k(\kappa(X))=\langle X,\nu\rangle^p \psi(X) \quad \text{for all $X \in \Sigma$}. \label{curvature measure equation}
\end{equation}

When $p=1$, this is exactly the equation for the prescribed curvature measure problem whose significance is discussed in Guan-Li-Li and Yang's papers \cite{Guan-Li-Li, Yang}, in which $C^2$ estimates have been obtained in $\bR^{n+1}$ and $\bH^{n+1}$, respectively. Following Lu, who obtained $C^2$-estimates for the wider range $p \in (-\infty,0) \cup (0,1]$ in $\bH^{n+1}$ \cite[Theorem 1.2]{Lu-CVPDE}, we extend their results not only to warped product manifolds and also to the same wider range of $p$'s.

\begin{theorem} \label{theorem curvature measure}
Let $1 \leq k \leq n$ and $p \in (-\infty,0) \cup (0,1]$. Suppose that $\Sigma$ is a closed strictly star-shaped $k$-convex hypersurface in a warped product manifold $\overline{M}$ satisfying the prescribed curvature measure type equation \eqref{curvature measure equation} for some positive function $\psi(X) \in C^2(\Sigma)$. Assume that the warping function of $\overline{M}$ is positive and has positive derivative. Then there exists some constant $C>0$ depending only on $n,k,p, \norm{M}_{C^1}, \inf \psi$, $\norm{\psi}_{C^2}$ and the curvature $\overline{R}$ of $\overline{M}$ such that
\[\max_{X \in \Sigma} \kappa_{\max}(X) \leq C.\]
\end{theorem}

\begin{remark}
We shall mention that in \cite{Chen}, Chen proved theorem \ref{theorem curvature measure} for $k=2$ in $\bR^{n+1}$ based on the study of optimal concavity of the $\sigma_2$ operator; and Huang-Xu \cite{Huang-Xu} proved theorem \ref{theorem curvature measure} for all $2 \leq k \leq n$ in $\bR^{n+1}$ by generalizing the method of Guan-Li-Li \cite{Guan-Li-Li}.
\end{remark}

By extending the lower order estimates in \cite{Guan-Lin-Ma, Huang-Xu, Yang} to warped product manifolds, the existence theorem follows.
\begin{theorem}\label{prescribed curvature measure}
If either (i) $1 \leq k \leq n-1, p \in (-\infty,0)\cup (0,1]$, or (ii) $k=n, p \in (-\infty,0)\cup (0,1)$, then there exists a unique $k$-convex star-shaped $C^{3,\alpha}$ hypersurface $\Sigma$ in $\overline{M}$ satisfying \eqref{curvature measure equation}. 

For $k=n$ and $p=1$, the same existence result would hold with the additional assumption that $\min_{M} \psi>1$.
\end{theorem}
\begin{remark}
In particular, this generalizes the main result \cite[Theorem 1.1 and Theorem 5.1]{Yang} of Yang from $p=1$ to $p \in (-\infty,0) \cup (0,1]$ in hyperbolic space. But still, our proof is highly inspired by that of Yang.
\end{remark}
\begin{remark}
The extra condition $\min \psi>1$ is required only when deriving the $C^0$ estimate in the case that $k=n$ and $p=1$. On the other hand, our $C^1$ estimate for \eqref{curvature measure equation} holds for all $p \neq 0$ without any barrier conditions; previously, this result was only obtained in $\bR^{n+1}$ \cite[Lemma 2.3]{Huang-Xu} and hence it is new in e.g. $\bH^{n+1}$.
\end{remark}

The ultimate goal is to know whether theorem \ref{theorem 2k>n} holds for $3 \leq k \leq n-3$; see remark 3.5 in \cite{Guan-Li-Li}. This is a rather difficult task and it still remains open up to this date. 

Recall that in their groundbreaking paper, Guan-Ren-Wang proved this for convex solutions and they claimed in a remark \cite[Remark 4.7]{Guan-Ren-Wang} that their proof would also work for semi-convex and ($k+1$)-convex solutions, by modifying their test function \eqref{Guan-Ren-Wang test function}. We provide a simpler proof for this claim using our test function \eqref{my test function} which will make a concavity inequality of Lu in \cite{Lu-CVPDE} applicable and we extend this result to warped product manifolds.

\begin{theorem} \label{weak convexity}
Let $3 \leq k \leq n-3$. Suppose that $\Sigma$ is a closed strictly star-shaped $k$-convex hypersurface in a warped product manifold $\overline{M}$ satisfying the curvature equation \eqref{curvature equation 1} for some positive function $\Psi(X,\nu) \in C^2(\Gamma)$, where $\Gamma$ is an open neighborhood of the unit normal bundle of $\Sigma$ in $\overline{M} \times \bS^n$. 
Assume that the warping function of $\overline{M}$ is positive and has positive derivative.
If either one of the following conditions
\begin{enumerate}[itemsep=5pt]
\item[(a)] $\Sigma$ is semi-convex i.e. there exists some $\eta>0$ such that for all $X \in M$ and all $1 \leq i \leq n$, we have $\kappa_i(X) \geq -\eta$.
\item[(b)] $\Sigma$ is $(k+1)$-convex i.e. $\kappa(X) \in \Gamma_{k+1}$.
\end{enumerate} holds, then there exists some constant $C>0$ depending only on $n,k, \eta,  \norm{M}_{C^1}, \inf \Psi$, $\norm{\Psi}_{C^2}$ and the curvature $\overline{R}$ of $\overline{M}$ such that
\[\max_{X \in \Sigma} \kappa_{\max}(X) \leq C.\]
\end{theorem}

\begin{remark}
Lu \cite{Lu-CVPDE} proved the semi-convex case in $\bH^{n+1}$ based on the method of Yang \cite{Yang} and that of Guan-Ren-Wang \cite{Guan-Ren-Wang}; our proof closely follows theirs, and in fact, it is simply a modest extension of Lu's proof with minor simplifications.

Note also that Lu's proof would work for the ($k+1$)-convex case as well and that was how we proved the $(k+1)$-convex case in our initial draft. Later, Lu pointed out to us that, according to a lemma of Li-Ren-Wang \cite{Li-Ren-Wang}, $(k+1)$-convexity implies semi-convexity. Therefore, it is sufficient to solve the semi-convex case only. We also include a version of the Li-Ren-Wang lemma in the context of our curvature equation; see lemma \ref{Li-Ren-Wang} below.
\end{remark}

The rest of this note is organized as follows. In section \ref{preliminaries}, we list auxiliary facts such as geometric formulas about hypersurfaces in warped product manifolds and properties of the $\sigma_k$ operator. In section \ref{C^2 estimates 1}, we perform a formal computation for a combined test function, which will be quoted in later sections. In section \ref{C^2 estimates 2}, we prove theorem \ref{theorem 2k>n}, theorem \ref{theorem n-2} and theorem \ref{weak convexity}. Finally in section \ref{C^2 estimates 3}, we derive $C^0, C^1$ and $C^2$ estimates for the prescribed curvature measure type equation \eqref{prescribed curvature measure}, which will imply theorem \ref{theorem curvature measure} and theorem \ref{prescribed curvature measure}.

\subsection*{Acknowledgement}
We would like to thank Professor Siyuan Lu for introducing the papers \cite{Ren-Wang-1, Ren-Wang-2, Chen-Li-Wang} of Ren-Wang and Chen-Li-Wang to us. This work is also inspired by Lu's recent two papers \cite{Lu-PAMS, Lu-CVPDE}. We would also like to thank him for pointing out a lemma due to Li-Ren-Wang in \cite{Li-Ren-Wang} i.e. our lemma \ref{Li-Ren-Wang}, which simplifies our proof further.

\textit{Note added on May 8th, 2024.} 

One week after we posted version 1 of this \href{https://arxiv.org/abs/2404.19562}{preprint} on arXiv, we found that Chen-Tu-Xiang posted a \href{https://arxiv.org/abs/2405.03407}{paper} \cite{Chen-Tu-Xiang} in which they proved theorem \ref{theorem 2k>n} independently. However, their test functions and also their arguments are different from ours, though the central ideas are the same i.e. applying the maximum principle to the largest principal curvature. Their proof is more in the style similar to Chu \cite{Chu} (which is very interesting!), while ours follow more closely that of Lu \cite{Lu-PAMS, Lu-CVPDE}. This provides new perspectives in deriving global curvature estimates.

\section{Preliminaries} \label{preliminaries}
In this section, we collect facts and formulas about hypersurfaces in warped product manifolds, which can all be found in \cite{Chen-Li-Wang} or \cite{DBD}; here we are merely listing these information for completeness; also to fix notations.

Let $(M,g')$ be a compact Riemannian manifold and let $I$ be an interval in $\bR$. For a smooth positive function $\phi:I \to \bR$ with $\phi'>0$, we define a warped product manifold $\overline{M}=I \times_{\phi} M$ endowed with the metric
\[\overline{g}=ds^2=dr^2+\phi(r)^2g', \quad r \in I\] which will also be simply denoted by $\langle \cdot,\cdot\rangle$.

\begin{remark}
Note that when $M=\bS^n$,
\[\overline{M}=\begin{cases}
\bR^{n+1}, & \text{if $\phi(r)=r$ and $I=[0,\infty)$}\\
\bH^{n+1}, & \text{if $\phi(r)=\sinh(r)$ and $I=[0,\infty)$}\\
\bS^{n+1}, & \text{if $\phi(r)=\sin(r)$ and $I=[0,\pi/2)$}
\end{cases}.\]
\end{remark}
The Riemannian connection on $\overline{M}$ will be denoted by $\overline{\nabla}$ and the connection on $M$ will be denoted by $\nabla'$. The same convention applies to the curvature tensors $\overline{R}$ and $R'$.

Let $\{e_1,\ldots,e_{n-1}\}$ be an orthonormal frame in $M$ and let $\{\theta_i\}$ be the associated dual frame. An orthonormal frame in $\overline{M}$ can then be defined as $\overline{e}_i:=(1/\phi)e_i$, $1 \leq i \leq n-1$ and $\overline{e}_0:=\partial/\partial r$; the associated dual frame is then $\overline{\theta}_i:=\phi\theta_i$, $1 \leq i \leq n-1$ and $\overline{\theta}_0=dr$.

We may represent a starshaped hypersurface $\Sigma$ as a radial graph of a differentiable function $r: M \to I$ over $M$ i.e. 
\[\Sigma=\{X(z)=(r(z),z): z \in M\},\] whose tangent space is spanned at each point by the vectors
\[X_i=\phi \overline{e}_i+r_i\overline{e}_0,\] where $r_i$'s are the components of the differential $dr=r_i\theta^i$.

The unit outward \footnote{In \cite{Chen-Li-Wang}, Chen-Li-Wang used the inward normal; the reader shall note the sign difference.} normal is given by 
\[\nu=\frac{1}{\sqrt{\phi^2+|\nabla'r|^2}}\left(\phi\overline{e}_0-\sum_{i}r^i\overline{e}_i\right);\] here $|\nabla'r|^2=r^ir_i$ is the squared norm of $\nabla'r=r^ie_i$.

The induced metric on $\Sigma$ is given by
\[g_{ij}=\langle X_i,X_j\rangle=\phi^2\delta_{ij}+r_ir_j \quad \text{with inverse} \quad g^{ij}=\frac{1}{\phi^2}\left(\delta_{ij}-\frac{r^ir^j}{\phi^2+|\nabla'r|^2}\right).\]

The second fundamental form of $\Sigma$ is given by
\[h_{ij}=\langle \overline{\nabla}_{X_j}X_i,\nu\rangle=\frac{1}{\sqrt{\phi^2+|\nabla'r|^2}}(-\phi r_{ij}+2\phi'r_ir_j+\phi^2\phi'\delta_{ij}),\] where $r_{ij}$ are the components of the Hessian $\nabla'^2r=\nabla' dr$ of $r$ in $M$.

\begin{lemma}\label{commutator formula}
Let $X_0$ be a point of $\Sigma$ and let $\{E_0=\nu, E_1,\ldots,E_n\}$ be an adapted frame field such that each $E_i$ is a principal direction and the associated dual frame satisfies $\omega_{i}^{k}=0$ at $X_0$. Then at $X_0$, we have
\[\nabla_{k}h_{ij}=\nabla_{j}h_{ik}+\overline{R}_{0ijk}\] and
\[h_{ii11}-h_{11ii}=h_{11}h_{ii}^2-h_{11}^2h_{ii}+2(h_{ii}-h_{11})\overline{R}_{i1i1}+h_{11}\overline{R}_{i0i0}-h_{ii}\overline{R}_{1010}+\overline{R}_{i1i0;1}-\overline{R}_{1i10;i}.\]
\end{lemma}
\begin{remark}
The frame field $E_i$ may be obtained from the adapted frame $\nu, X_1,\ldots,X_n$ by the Gram-Schmidt procedure. Since this last frame depends only on $r$ and $\nabla' r$, we may conclude that the components of $\overline{R}$ and $\bar{\nabla} \overline{R}$ calculated in  terms of the frame $E_i$ depend only on $r$ and $\nabla' r$.
\end{remark}
\begin{proof}
See lemma 2.1 in \cite{DBD}.
\end{proof}

$\overline{M}$ naturally comes with a conformal Killing position vector field $V=\phi(r)\partial_r$ and we denote by $\nu(V)$ the outward unit normal. That is, we will study the prescribed curvature equation in the following form:
\begin{equation}
\sigma_{k}(\kappa[\Sigma])=\Psi(V,\nu). \label{curvature equation 2}
\end{equation}

We also define two auxiliary quantities: the support function $u:=\langle V,\nu\rangle$ and 
\[\Phi(r):=\int_{0}^{r} \phi(\rho)\ d\rho.\]

\begin{lemma} \label{geometric formulas}
\begin{align*}
\nabla_{E_i}\Phi&=\phi\langle \overline{e}_0,E_i\rangle E_i\\
\nabla_i u &= g^{kl}h_{ik}\nabla_{E_l}\Phi \\
\nabla_{E_i,E_j}^2 \Phi &= \phi' g_{ij}-uh_{ij}\\
\nabla_{ij}u&=g^{kl}\left(\nabla_{E_k}h_{ij}-\overline{R}_{0ijk}\right)\nabla_{E_l}\Phi+\phi'h_{ij}-ug^{kl}h_{ik}h_{jl}
\end{align*}
\end{lemma}
\begin{proof}
See proposition 2.4 in \cite{Guan-Li-Wang}, lemma 4.1 in \cite{DBD} and lemma 2.3 in \cite{Chen-Li-Wang}.
\end{proof}

Before we end this section, we list a few useful properties of the $\sigma_k$ operator. The reader is also referred to \cite{Spruck-MSRI, Wang} and \cite[Chapter XV]{Lieberman} for more.

The $k$-th elementary symmetric polynomial $\sigma_k : \bR^n \to \bR$ is a smooth symmetric function of $n$ variables, defined by 
\[\sigma_{k}(\kappa)=\sum_{i_1<i_2<\cdots<i_k} \kappa_{i_1}\cdots\kappa_{i_k}\] and the \Garding\ cones are defined by 
\[\Gamma_{k}=\{\kappa \in \bR^n: \sigma_{j}(\kappa)>0 \quad \forall\ 1 \leq j \leq k\}.\]

\begin{definition} \label{k-convex}
An embedded hypersurface $\Sigma$ in $\overline{M}$ is said to be a \textbf{$k$-convex} or \textbf{$k$-admissible} solution to the curvature equation
\[\sigma_k(\kappa[\Sigma])=\Psi(V,\nu)\]
if its principal curvatures $\kappa=(\kappa_1,\ldots,\kappa_n)$ belong to the $k$th \Garding 's cone $\Gamma_k$ at every point.
\end{definition}

\begin{notation}
Observe that 
\[\frac{\partial \sigma_k}{\partial \kappa_i} = \sigma_{k-1}(\kappa_1,\ldots,\kappa_{i-1},0,\kappa_i,\ldots,\kappa_n)=\sigma_{k-1}(\kappa)\bigg|_{\kappa_i=0}\] and we introduce the notation
\[\sigma_{k}^{ii}, \quad \sigma_{k-1}(\kappa|i), \quad \text{or} \quad \sigma_{k-1;i}(\kappa)\] to mean the same thing.

Similarly, for the second order derivatives, we use the following notations
\[\sigma_{k}^{ii,jj}=\sigma_{k-2}(\kappa|ij)=\sigma_{k-2;ij}(\kappa)=\frac{\partial^2\sigma_k}{\partial \kappa_i \partial \kappa_j}.\]
\end{notation}
\begin{lemma} \label{sigma_k formulas}
For $\kappa \in \bR^n$, we have the following
\begin{align*}
\sigma_k(\kappa)=\kappa_i\sigma_{k-1}(\kappa|i)+\sigma_k(\kappa|i), \quad 
\sum_{i=1}^{n} \kappa_i \sigma_{k-1}(\kappa|i)=k\sigma_k, \quad
\sum_{i=1}^{n} \sigma_{k-1}(\kappa|i)=(n-k+1)\sigma_{k-1}
\end{align*}
\end{lemma}
\begin{lemma}\label{sigma_l}
For $\kappa \in \Gamma_{k}$ and $1 \leq l < k$, we have 
\[\sigma_{l}(\kappa)>\kappa_1\cdots \kappa_l\] and 
\[\sigma_{k}(\kappa) \leq C\kappa_1 \cdots \kappa_k.\]
\end{lemma}
\begin{proof}
For the first inequality, see lemma 12 in \cite{Ren-Wang-2}; an explicit proof is given in the preprint version of that paper. For the second inequality, see lemma A.1 in \cite{Li-1991}.
\end{proof}
\begin{lemma}\label{negative kappa}
For $\kappa=(\kappa_1,\ldots,\kappa_n) \in \Gamma_k$, if $\kappa_i \leq 0$, then 
\[-\kappa_i<\frac{n-k}{k}\kappa_1.\]
\end{lemma}
\begin{proof}
See lemma 11 in \cite{Ren-Wang-2}.
\end{proof}
\begin{lemma} \label{kappa squared times f_i}
Let $1 \leq k \leq n$.
If $\kappa = (\kappa_1,\ldots,\kappa_n) \in \Gamma_{k}$ is ordered as $\kappa_1 \geq \kappa_2 \geq \cdots \geq \kappa_n$, then we have 
\[\sum_{i=1}^{n} \kappa_{i}^2\sigma_{k}^{ii} \geq \frac{k}{n}\kappa_1\sigma_k.\]
\end{lemma}
\begin{proof}
When $k=1$, the inequality follows by noting that $\sigma_{1}^{ii}=1$ and $\kappa_1 \geq \frac{1}{n} \sigma_1$. When $2 \leq k \leq n$, by lemma 2.3 in \cite{Lu-PAMS}, we have $\sum \kappa_{i}^2 \sigma_{k}^{ii} \geq \frac{k}{n}\sigma_1\sigma_k$. Noting that $\sigma_1-\kappa_1=\frac{\partial \sigma_2}{\partial \kappa_i}>0$ for $\kappa \in \Gamma_k \subseteq \Gamma_2$ yields the inequality.
\end{proof}

Finally, we have the concavity inequalities due to Ren-Wang and Lu.

\begin{lemma} \label{Ren-Wang}
Let $k=n-1, n \geq 3$ or $k=n-2, n\geq 5$, and let $\xi \in \bR^n$ be an arbitrary vector. Suppose that $\kappa \in \Gamma_{k}$ is ordered as $$\kappa_1 >\kappa_2 \geq \kappa_3 \geq \cdots \geq \kappa_n.$$ Suppose also that $N_0 \leq \sigma_{k}(\kappa) \leq N_1$. If $\kappa_1$ is sufficiently large, then there exists some $\beta>0$ such that 
\[\kappa_1\left[\beta\left(\sum_{i=1}^{n} \sigma_{k}^{ii}\xi_{i}\right)^{2}-\sum_{p \neq q} \sigma_{k}^{pp,qq}\xi_p\xi_q\right]-\sigma_{k}^{11}\xi_{1}^2+\sum_{i\neq 1} \frac{2\kappa_1}{\kappa_1-\kappa_i}\sigma_{k}^{ii}\xi_{i}^2 \geq 0.\]
\end{lemma}
\begin{proof}
We first prove for $k=n-2$, $n \geq 5$. By theorem 4 in \cite{Ren-Wang-2}, we have that 
\[\kappa_1\left[\beta\left(\sum_{i=1}^{n} \sigma_{k}^{ii}\xi_{i}\right)^{2}-\sum_{p \neq q} \sigma_{k}^{pp,qq}\xi_p\xi_q\right]-\sigma_{k}^{11}\xi_{1}^2+\sum_{i\neq 1} a_{i}\xi_{i}^2 \geq 0\] where 
\[a_i=\sigma_{k}^{ii}+(\kappa_1+\kappa_i)\sigma_{k}^{11,ii}.\]

Recall that 
\[\sigma_{k}^{11,ii}=\frac{\sigma_{k}^{ii}-\sigma_{k}^{11}}{\kappa_1-\kappa_i}\] and so we have 
\begin{align*}
a_{i}&=\sigma_{k}^{ii}+(\kappa_1+\kappa_i)\sigma_{k}^{11,ii} = \sigma_{k}^{ii} + \frac{\kappa_1+\kappa_i}{\kappa_1-\kappa_i}(\sigma_{k}^{ii}-\sigma_{k}^{11}) \\
&\leq \sigma_{k}^{ii}+\frac{\kappa_1+\kappa_i}{\kappa_1-\kappa_i}\sigma_{k}^{ii} \\
&=\frac{2\kappa_1}{\kappa_1-\kappa_i}\sigma_{k}^{ii}.
\end{align*}

The result then follows. Note that we have used lemma \ref{negative kappa} with $2k>n$ to get 
\[\frac{\kappa_1+\kappa_i}{\kappa_1-\kappa_i} \geq 0.\]

For $k=n-1$ with $n \geq 3$, we apply theorem 11 in \cite{Ren-Wang-1} to have that
\[\kappa_1\left[\beta\left(\sum_{i=1}^{n} \sigma_{n-1}^{ii}\xi_{i}\right)^{2}-\sum_{p \neq q} \sigma_{n-1}^{pp,qq}\xi_p\xi_q\right]-\sigma_{n-1}^{11}\xi_{1}^2+(1+\varepsilon)\sum_{i\neq 1} \sigma_{n-1}^{ii}\xi_{i}^2 \geq 0\] where $\varepsilon>0$ is arbitrary and $\beta$ depends on $\varepsilon$.

By lemma \ref{negative kappa} again, we can choose $\varepsilon$ small enough e.g. $\varepsilon=[(2n-2)/n]-1$ so that
\[\frac{2\kappa_1}{\kappa_1-\kappa_i}\geq \frac{2n-2}{n}=1+\varepsilon, \quad n \geq 3 \quad \text{and} \quad i \neq 1.\] The result then follows.
\end{proof}

\begin{lemma} \label{Lu}
Suppose $\kappa=(\kappa_1,\ldots,\kappa_n) \in \Gamma_k$ is ordered as $\kappa_1 \geq \cdots \geq \kappa_n$. Let $\xi \in \bR^n$ be an arbitrary vector and $1 \leq l<k$. Given $\varepsilon,\delta, \delta_0 \in (0,1)$, we can find some $\delta'>0$ such that if $\kappa_l \geq \delta \kappa_1$ and $\kappa_{l+1} \leq \delta'\kappa_1$, then 
\begin{equation}
-\sum_{p \neq q} \frac{\sigma_{k}^{pp,qq}\xi_{p}\xi_{q}}{\sigma_k}+\frac{\left(\sum_{i}\sigma_{k}^{ii}\xi_i\right)^2}{\sigma_{k}^2} \geq (1-\varepsilon)\frac{\xi_{1}^2}{\kappa_{1}^2}-\delta_0\sum_{i>l}\frac{\sigma_{k}^{ii}\xi_{i}^2}{\kappa_1\sigma_{k}}. \label{Lu's inequality}
\end{equation}
\end{lemma}
\begin{proof}
See lemma 3.1 in \cite{Lu-CVPDE}.
\end{proof}

Finally, we demonstrate the fact that $(k+1)$-convexity implies semi-convexity, which is a lemma proved by Li-Ren-Wang in \cite{Li-Ren-Wang}, for solutions of the Hessian equation $\sigma_k(D^2u)=f(x,u,Du)$. Here we prove the corresponding version for our curvature equation \eqref{curvature equation 1}.
\begin{lemma} \label{Li-Ren-Wang}
Suppose that $\Sigma$ is a $(k+1)$-convex solution to equation \eqref{curvature equation 1}. Then there exists some constant $\eta>0$ depending on known data of $\Sigma$ and $\Psi$ such that 
\[\kappa_i(X) \geq -\eta \quad \text{for all $1 \leq i \leq n$ and all $X \in \Sigma$}.\]
\end{lemma}
\begin{proof}
The proof is adapted from \cite[Lemma 7]{Li-Ren-Wang}. We choose $\eta$ such that 
\[\left(\frac{\eta}{n}\right)^k \geq \sup_{X \in \Sigma} \Psi(X,\nu).\] Let $\kappa_1 \geq \kappa_2 \geq \cdots \kappa_n$ denote the principal curvatures of $\Sigma$. Since $\Sigma$ is  $(k+1)$-convex, we have that $\kappa=(\kappa_1,\ldots,\kappa_n) \in \Gamma_{k+1}$. Hence by using lemma \ref{sigma_l}, we have that 
\[\sigma_k(\kappa) \geq \kappa_1\cdots\kappa_k \geq \kappa_{k}^{k}.\] It follows that 
\[\kappa_k \leq \frac{\eta}{n}.\]

Also, using $\kappa \in \Gamma_k$, we have that the sum of any $(n-k+1)$ components of $\kappa$ is positive. In particular, we have
\[\sum_{i=k}^{n} \kappa_i>0\] which implies that 
\begin{align*}
0<\kappa_k+\kappa_{k+1}+\cdots+\kappa_n \leq (n-k)\kappa_k+\kappa_n \leq \frac{n-k}{n}\eta+\kappa_n
\end{align*} i.e.
\[\kappa_n+\eta \geq 0.\]
\end{proof}
\section{A Model Computation} \label{C^2 estimates 1}
In this section, we establish preparatory work for obtaining the following curvature estimate, which will imply theorem \ref{theorem 2k>n}, theorem \ref{theorem n-2}, theorem \ref{theorem curvature measure},  and theorem \ref{weak convexity}.
\begin{theorem} \label{curvature estimate}
Let $\Sigma=\{(r(z),z): z \in M\}$ be a closed $k$-convex hypersurface in $\overline{M}$ which is strictly star-shaped with respect to the origin and satisfies equation \eqref{curvature equation 2} for some positive function $\Psi(V,\nu) \in C^2(\Gamma)$, where $\Gamma$ is an open neighborhood of the unit normal bundle of $\Sigma$ in $\overline{M} \times \bS^n$.

Suppose that we have uniform control $0<r_1<r(z)<r_2<b$ and $|r|_{C^1} \leq r_3$. If additionally either of the following holds:
\begin{itemize}[itemsep=5pt]
\item $k=n-1, n \geq 3$ or $k=n-2, n \geq 5$,
\item $\Sigma$ is semi-convex i.e. there exists some $\eta>0$ such that for all $X \in \Sigma$ and all $1 \leq i \leq n$, we have $\kappa_{i}(X) \geq -\eta$.
\item $\Sigma$ is $(k+1)$-convex i.e. $\kappa \in \Gamma_{k+1}$, or 
\item $\Phi(V,\nu)=\langle V,\nu\rangle^p \psi(V)$ for $p \in (-\infty,0)\cup (0,1]$ and $\psi \in C^2(\Sigma)$,
\end{itemize} then there exists some $C>0$ depending on $n,k,p, r_1, r_2, r_3, \inf \Psi$, $\norm{\Psi}_{C^2}$ and $\overline{R}$ such that 
\[\max_{z \in M} |\kappa_{\max}(z)| \leq C.\]
\end{theorem}

The proof for different cases will be different. More specifically, for the first three cases, we will use the test function
\[Q=\log \kappa_{\max}-N\log u + \alpha \Phi\] with $N,\alpha>0$ being some large numbers, in order to make the concavity inequalities in lemma \ref{Ren-Wang} and lemma \ref{Lu} applicable.

In contrast, for the fourth case, we will be working with the test function
\[Q=\log \kappa_{\max}-\log(u-a)+\alpha \Phi\] without the large number $N$ and where $a>0$ satisfies $u \geq 2a$. This is because $N$ would have produced a negative term with the coefficient $-N$ which are difficult to handle without a handy concavity inequality, if we would like $N$ to be large. In this case, we have to make delicate use of the special structure of the right-hand side $\Psi(X,\nu)=u^p\psi(X)$ with the help of the gauge term $a$.

However, a large part of the derivation will be almost exactly the same in all the four cases. Therefore, in this section, we will do a "model computation" for the combined test function
\[Q=\log \kappa_{\max}-N\log(u-a)+\alpha \Phi.\] The purpose for doing so is to reduce the labor and shorten the proof. More precisely, we will substitute $a=0$ in section \ref{C^2 estimates 2} and $N=1$ in section \ref{C^2 estimates 3}, then we can proceed from there and thus redundant calculations will be eliminated.

Now we begin the calculations. Define 
\[Q(z,\xi)=\log \sff(\xi,\xi)-N\log(u-a)+\alpha \Phi\] where $\xi$ is a unit tangent vector to $\Sigma$ at $X=(r(z),z)$ and $\sff$ is the second fundamental form.
Suppose that $Q$ attains its maximum at $X_0=(r(z_0),z_0)$. We may choose a local orthonormal frame $\{E_1,\ldots, E_n\}$ such that at this point we have $\xi=E_1$ and $h_{ij}=\kappa_i\delta_{ij}$, where the principal curvatures are ordered as $\kappa_1 \geq \kappa_2 \geq \cdots \geq \kappa_n$. 

\begin{remark}
Technically speaking, our operation here is not really valid; because the point $X_0$ and hence the local orthonormal frame $\{e_i\}$ depend on the choice of the test function $Q$.

Here we are merely doing a symbolic computation formally! This is for not to repeat this part of calculation in section \ref{C^2 estimates 2} and section \ref{C^2 estimates 3}.
\end{remark}

Now, note that if $\kappa_1$ has multiplicity more than one i.e.
\[\kappa_1=\kappa_2=\cdots=\kappa_m > \kappa_{m+1} \geq \cdots \geq \kappa_n \quad \text{for some $m>1$},\]
then $Q$ is not smooth at $X_0$. To resolve this issue, we apply a standard perturbation argument following Chu \cite{Chu}. Let $g$ be the first fundamental form of $\Sigma$. Near $X_0$, we define a new tensor $B$ by 
\[B(V_1,V_2)=g(V_1,V_2)-g(V_1,E_1)g(V_2,E_1)\] for tangent vectors $V_1$ and $V_2$. Denote by $B_{ij}=B(E_i,E_j)$ and it is clear that $B_{ij}=\delta_{ij}(1-\delta_{1i})$. We now define a new matrix by $\tilde{h}_{ij}:=h_{ij}-B_{ij}$ whose eigenvalues are ordered as 
\[\tilde{\kappa}_1 \geq \tilde{\kappa}_{2} \geq \cdots \geq \tilde{\kappa}_n.\] Note that $\kappa_1 \geq \tilde{\kappa}_{1}$ near $X_0$ and 
\[\tilde{\kappa}_i=\begin{cases}
\kappa_1, & i =1 \\
\kappa_i-1, & i>1
\end{cases}\] at $X_0$. It then follows that $\tilde{\kappa}_1$ has multiplicity one and hence that the perturbed test function 
\[\widetilde{Q}:=\log \tilde{\kappa}_1 - N\log (u-a) + \alpha\Phi\] is smooth at $X_0$. Below we will proceed to perform a lengthy and tedious computation at the point $X_0$; the reader may start from \eqref{final inequality}.

\begin{notation}
In our derivation below, the symbol $C$ will denote some controlled quantity whose magnitude depends on some known data. Its value may change from line to line without explicitly saying so, but it will still be denoted by the same symbol. This would not cause confusion because its magnitude is not of relevance in our analysis.

Also, whenever we say the vague phrase "$\kappa_1$ is sufficiently large", it would mean that given any such a constant $C>0$, we must have $\kappa_1 > C$ otherwise the desired curvature bound $\kappa_1 \leq C$ would have been obtained already.

We also denote 
\[F^{ij}=\frac{\partial F}{\partial h_{ij}}, \quad F^{ij,kl}=\frac{\partial^2 F}{\partial h_{ij}h_{kl}}.\]
\end{notation}
Using $\tilde{h}_{ij}=h_{ij}-B_{ij}$, we have that at $X_0$,
\begin{align*}
(\tilde{\kappa}_{1})_{i}&=\tilde{h}_{11i}=h_{11i},\\
(\tilde{\kappa}_1)_{ii} & = \tilde{h}_{11ii}+2\sum_{p \neq 1}\frac{\tilde{h}_{1pi}^2}{\tilde{\kappa}_1-\tilde{\kappa}_p}= h_{11ii}+2\sum_{p\neq 1}\frac{h_{1pi}^2}{\kappa_1-\tilde{\kappa}_p}
\end{align*} and so 
\begin{align}
0&=\frac{(\tilde{\kappa}_{1})_i}{\tilde{\kappa}_1}-N\frac{u_i}{u-a}+\alpha \Phi_i = \frac{h_{11i}}{\kappa_1}-N\frac{u_i}{u-a}+\alpha\Phi_i, \label{1st critical 1} \\
0&\geq \frac{(\tilde{\kappa}_{1})_{ii}}{\tilde{\kappa}_1}-\frac{(\tilde{\kappa}_1)_{i}^2}{\tilde{\kappa}_{1}^2}-N\frac{u_{ii}}{u-a}+N\frac{u_{i}^2}{(u-a)^2}+\alpha \Phi_{ii} \nonumber \\
&=\frac{h_{11ii}}{\kappa_1}+2\sum_{p\neq 1}\frac{h_{1pi}^2}{\kappa_1(\kappa_1-\tilde{\kappa}_p)}-\frac{h_{11i}^2}{\kappa_{1}^2}-N\frac{u_{ii}}{u-a}+N\frac{u_{i}^2}{(u-a)^2}+\alpha \Phi_{ii}. \label{2nd critical 1}
\end{align}

Contracting \eqref{2nd critical 1} with $F=\sigma_k$, we obtain that 
\begin{gather}\label{contract with F}
\begin{split}
0 &\geq \sum_{i} \frac{F^{ii}h_{11ii}}{\kappa_1}+2\sum_{i}\sum_{p \neq 1}\frac{F^{ii}h_{1pi}^2}{\kappa_1(\kappa_1-\tilde{\kappa}_p)}-\sum_{i}\frac{F^{ii}h_{11i}^2}{\kappa_{1}^2} \\
&\quad -N\sum_{i} \frac{F^{ii}u_{ii}}{u-a}+N\sum_{i}\frac{F^{ii}u_{i}^2}{(u-a)^2}+\alpha \sum_{i} F^{ii}\Phi_{ii}. 
\end{split}
\end{gather}

By the commutator formula (see lemma \ref{commutator formula}),
\begin{align*}
\nabla_{11}h_{ii}=&\ \nabla_{ii}h_{11}+h_{11}h_{ii}^2-h_{11}^2h_{ii}+2(h_{ii}-h_{11})\overline{R}_{i1i1} \\
&+h_{11}\overline{R}_{i0i0}-h_{ii}\overline{R}_{1010}+\overline{R}_{i1i0;1}-\overline{R}_{1i10;i}
\end{align*}
from which it follows that for some $C>0$ depending on $k, \overline{R}$ and $\Psi$
\begin{equation}
F^{ii}h_{11ii} \geq F^{ii}h_{ii11}-\kappa_1\sum_i F^{ii} \kappa_{i}^2+\kappa_{1}^2k\Psi-C\kappa_1\sum_i F^{ii} - C \label{1st term}.
\end{equation} where we have used lemma \ref{sigma_k formulas} to get $\sum F^{ii}\kappa_i=kF=k\Psi$.

Also, for the second term in \eqref{contract with F}, we have 
\begin{align}
2\sum_{i}\sum_{p \neq 1}\frac{F^{ii}h_{1pi}^2}{\kappa_1(\kappa_1-\tilde{\kappa}_p)}&\geq 2\sum_{p\neq 1}\frac{F^{pp}h_{1pp}^2}{\kappa_1(\kappa_1-\tilde{\kappa}_p)}+2\sum_{p\neq 1}\frac{F^{11}h_{1p1}^2}{\kappa_1(\kappa_1-\tilde{\kappa}_p)}. \label{2nd term}
\end{align}

Now, we need to be cautious that the second fundamental form is no longer Codazzi and extra care with the third order terms must be involved. Recall from lemma \ref{commutator formula}, we have that
\begin{equation}
h_{ijk}=h_{ikj}+\overline{R}_{0ijk}. \label{Codazzi}
\end{equation} Since there will also be lots of squaring operations, we introduce a convenient algebraic lemma from \cite{Hou-Ma-Wu}: for any $0 < \varepsilon < 1$, we have
\begin{equation}
|a+b|^2 \geq \varepsilon |a|^2-\frac{\varepsilon}{1-\varepsilon}|b|^2, \quad a,b \in \bC^n. \label{HMW}
\end{equation}

By inserting \eqref{1st term} and \eqref{2nd term} into \eqref{contract with F}, we have
\begin{gather} \label{2nd critical 2}
\begin{split}
0&\geq \sum_i \frac{F^{ii}h_{ii11}}{\kappa_1}-\sum_i \kappa_{i}^2F^{ii}+\kappa_1\ k\Psi-C\sum F^{ii} - C\\
&\quad +2\sum_{p \neq 1}\frac{F^{ii}h_{1pp}^2}{\kappa_1(\kappa_1-\tilde{\kappa}_p)}+2\sum_{p\neq 1}\frac{F^{11}h_{1p1}^2}{\kappa_1(\kappa_1-\tilde{\kappa}_p)}-\sum_{i}\frac{F^{ii}h_{11i}^2}{\kappa_{1}^2}\\
&\quad -N\sum_{i} \frac{F^{ii}u_{ii}}{u-a}+N\sum_{i}\frac{F^{ii}u_{i}^2}{(u-a)^2}+\alpha\sum_{i} F^{ii}\Phi_{ii}
\end{split}
\end{gather}

Now let us evaluate the third line. By lemma \ref{geometric formulas}, we have
\begin{align*}
u_{ii}&=\sum_{k} (h_{iik}-\overline{R}_{0iik})\Phi_{k}+\phi'\kappa_i-\kappa_{i}^2u\\
\Phi_{ii}&=\phi'-u\kappa_i \\
\end{align*} and hence \eqref{2nd critical 2} becomes
\begin{gather} \label{2nd critical 3}
\begin{split}
0&\geq \sum_i \frac{F^{ii}h_{ii11}}{\kappa_1}-\sum_i F^{ii}\kappa_{i}^2+\kappa_1\ k\Psi-C\sum F^{ii} - C+N\sum_{i}\frac{F^{ii}u_{i}^2}{(u-a)^2}\\
&\quad +2\sum_{p \neq 1}\frac{F^{ii}h_{1pp}^2}{\kappa_1(\kappa_1-\tilde{\kappa}_p)}+2\sum_{p\neq 1}\frac{F^{11}h_{1p1}^2}{\kappa_1(\kappa_1-\tilde{\kappa}_p)}-\sum_{i}\frac{F^{ii}h_{11i}^2}{\kappa_{1}^2}\\
&\quad -\frac{N}{u-a}\sum_{i}F^{ii}\left[\sum_{k} (h_{iik}-\overline{R}_{0iik})\Phi_{k}+\phi'\kappa_i-\kappa_{i}^2u\right]+\alpha \sum_{i} F^{ii}(\phi'-u\kappa_i) \\
&=\sum_i \frac{F^{ii}h_{ii11}}{\kappa_1}-\sum_i F^{ii}\kappa_{i}^2+\kappa_1\ k\Psi-C\sum F^{ii}-C +N\sum_{i}\frac{F^{ii}u_{i}^2}{(u-a)^2}\\
&\quad +2\sum_{p \neq 1}\frac{F^{ii}h_{1pp}^2}{\kappa_1(\kappa_1-\tilde{\kappa}_p)}+2\sum_{p\neq 1}\frac{F^{11}h_{1p1}^2}{\kappa_1(\kappa_1-\tilde{\kappa}_p)}-\sum_{i}\frac{F^{ii}h_{11i}^2}{\kappa_{1}^2}+\frac{N}{u-a}\sum F^{ii}\overline{R}_{0iik}\Phi_k\\
&\quad -\frac{N}{u-a}\sum_{k}\left(\sum_{i} F^{ii}h_{iik}\Phi_k\right)-\frac{Nk\phi'}{u-a}\Psi+\frac{Nu}{u-a}\sum_{i}F^{ii}\kappa_{i}^2 + \alpha \phi' \sum_{i} F^{ii} - \alpha k u \Psi \\
&\geq \left(\frac{Nu}{u-a}-1\right)\sum_{i}F^{ii}\kappa_{i}^2+(\alpha \phi' - C)\sum_{i} F^{ii}+\kappa_1k\Psi -C\alpha- C\\
&\quad +\sum_i \frac{F^{ii}h_{ii11}}{\kappa_1}-\frac{N}{u-a}\sum_{k}\left(\sum_{i} F^{ii}h_{iik}\Phi_k\right)+N\sum_{i}\frac{F^{ii}u_{i}^2}{(u-a)^2}\\
&\quad +2\sum_{p \neq 1}\frac{F^{ii}h_{1pp}^2}{\kappa_1(\kappa_1-\tilde{\kappa}_p)}+2\sum_{p\neq 1}\frac{F^{11}h_{1p1}^2}{\kappa_1(\kappa_1-\tilde{\kappa}_p)}-\sum_{i}\frac{F^{ii}h_{11i}^2}{\kappa_{1}^2} \\
\end{split}
\end{gather} where we have used lemma \ref{sigma_k formulas} to get $\sum_{i} F^{ii}\kappa_i=kF=k\Psi$.

Next, we proceed to handle the $h_{ii11}$ term. By differentiating \eqref{curvature equation 2} twice, we obtain that 
\begin{align}
F^{ii}h_{iik}&=\Psi_{V}(\nabla_{E_k}V)+h_{ks}\Psi_{\nu}(E_s) \label{differentiate once} \\
\text{and} \quad F^{ii}h_{ii11}&+F^{ij,kl}h_{ij1}h_{kl1}=\nabla_{11}\Psi \label{differentiate twice}
\end{align}
Thus, we have
\begin{gather} \label{2nd critical 4}
\begin{split}
0&\geq \left(\frac{Nu}{u-a}-1\right)\sum_{i}F^{ii}\kappa_{i}^2+(\alpha \phi' - C)\sum_{i} F^{ii}+kF\kappa_1-C\alpha-C\\
&\quad -\sum_{i,j,k,l}\frac{F^{ij,kl}h_{ij1}h_{kl1}}{\kappa_1} +\frac{\Psi_{11}}{\kappa_1} - \frac{N}{u-a} \sum_{k} \left(\sum_{i} F^{ii}h_{iik}\Phi_k\right)\\
&\quad +2\sum_{p \neq 1}\frac{F^{ii}h_{1pp}^2}{\kappa_1(\kappa_1-\tilde{\kappa}_p)}+2\sum_{p\neq 1}\frac{F^{11}h_{1p1}^2}{\kappa_1(\kappa_1-\tilde{\kappa}_p)}-\sum_{i}\frac{F^{ii}h_{11i}^2}{\kappa_{1}^2}+N\sum_{i}\frac{F^{ii}u_{i}^2}{(u-a)^2}.
\end{split}
\end{gather}

Next, by an inequality due to Andrews \cite{Andrews} and Gerhardt \cite{Gerhardt}, we have
\begin{align*}
-\sum_{i,j,k,l}\frac{F^{ij,kl}h_{ij1}h_{kl1}}{\kappa_1} &=-\sum_{p\neq q} \frac{F^{pp,qq}h_{pp1}h_{qq1}}{\kappa_1} + \sum_{p\neq q} \frac{F^{pp,qq}h_{pq1}^2}{\kappa_1} \\
&\geq -\sum_{p\neq q} \frac{F^{pp,qq}h_{pp1}h_{qq1}}{\kappa_1} + 2\sum_{i>m}\frac{1}{\kappa_1}\frac{F^{ii}-F^{11}}{\kappa_1-\tilde{\kappa}_i}h_{1i1}^2
\end{align*} and so \eqref{2nd critical 4} becomes

\begin{gather}\label{2nd critical 5}
\begin{split}
0&\geq \left(\frac{Nu}{u-a}-1\right)\sum_{i}F^{ii}\kappa_{i}^2+(\alpha \phi' - C)\sum_{i} F^{ii}+kF\kappa_1-C\alpha-C\\
&\quad -\sum_{p\neq q} \frac{F^{pp,qq}h_{pp1}h_{qq1}}{\kappa_1} + 2\sum_{i>m}\frac{1}{\kappa_1}\frac{F^{ii}-F^{11}}{\kappa_1-\tilde{\kappa}_i}h_{1i1}^2\\
&\quad +2\sum_{p \neq 1}\frac{F^{ii}h_{1pp}^2}{\kappa_1(\kappa_1-\tilde{\kappa}_p)}+2\sum_{p\neq 1}\frac{F^{11}h_{1p1}^2}{\kappa_1(\kappa_1-\tilde{\kappa}_p)}-\sum_{i}\frac{F^{ii}h_{11i}^2}{\kappa_{1}^2}+N\sum_{i}\frac{F^{ii}u_{i}^2}{(u-a)^2}\\
&\quad +\frac{\Psi_{11}}{\kappa_1} - \frac{N}{u-a} \sum_{k} \left(\sum_{i} F^{ii}h_{iik}\Phi_k\right)
\end{split}
\end{gather}

Regrouping terms, we get
\begin{gather}
\begin{split}
0&\geq \left(\frac{Nu}{u-a}-1\right)\sum_{i}F^{ii}\kappa_{i}^2+(\alpha \phi' - C)\sum_{i} F^{ii}+kF\kappa_1-C\alpha-C \\
&\quad -\sum_{p\neq q} \frac{F^{pp,qq}h_{pp1}h_{qq1}}{\kappa_1}-\frac{F^{11}h_{111}^2}{\kappa_{1}^2}+2\sum_{i \neq 1}\frac{F^{ii}h_{1ii}^2}{\kappa_1(\kappa_1-\tilde{\kappa}_i)}+N\frac{F^{11}u_{1}^2}{(u-a)^2}\\
&\quad +2\sum_{i>m}\frac{1}{\kappa_1}\frac{F^{ii}-F^{11}}{\kappa_1-\tilde{\kappa}_i}h_{1i1}^2+2\sum_{i \neq 1}\frac{F^{11}h_{1i1}^2}{\kappa_1(\kappa_1-\tilde{\kappa}_i)}-\sum_{i \neq 1}\frac{F^{ii}h_{11i}^2}{\kappa_{1}^2}+N\sum_{i \neq 1}\frac{F^{ii}u_{i}^2}{(u-a)^2}\\
&\quad +\frac{\Psi_{11}}{\kappa_1} - \frac{N}{u-a} \sum_{k} \left(\sum_{i} F^{ii}h_{iik}\Phi_k\right)
\end{split}
\end{gather}

For the third line, we use \eqref{Codazzi} and \eqref{HMW} to get that
\begin{align*}
&\quad 2\sum_{i>m}\frac{1}{\kappa_1}\frac{F^{ii}-F^{11}}{\kappa_1-\tilde{\kappa}_i}h_{1i1}^2+2\sum_{i \neq 1}\frac{F^{11}h_{1i1}^2}{\kappa_1(\kappa_1-\tilde{\kappa}_i)}-\sum_{i \neq 1}\frac{F^{ii}h_{11i}^2}{\kappa_{1}^2} \\
=&\quad 2\sum_{i>m}\frac{F^{ii}h_{1i1}^2}{\kappa_1(\kappa_1-\tilde{\kappa}_i)}+2\sum_{1<i\leq m} \frac{F^{11}h_{1i1}^2}{\kappa_1(\kappa_1-\tilde{\kappa}_i)}-\sum_{i \neq 1}\frac{F^{ii}h_{11i}^2}{\kappa_{1}^2}\\
= &\quad 2\sum_{i>m}\frac{F^{ii}(h_{11i}+\overline{R}_{01i1})^2}{\kappa_1(\kappa_1-\tilde{\kappa}_i)}-\sum_{i > m}\frac{F^{ii}h_{11i}^2}{\kappa_{1}^2}\\
&\quad 2\sum_{1 \leq i \leq m} \frac{F^{11}(h_{11i}+\overline{R}_{01i1})^2}{\kappa_1(\kappa_1-\tilde{\kappa}_i)}-\sum_{1 < i \leq m} \frac{F^{ii}h_{11i}^2}{\kappa_{1}^2} \\
\geq &\quad 2\varepsilon \sum_{i>m} \frac{F^{ii}h_{11i}^2}{\kappa_1(\kappa_1-\tilde{\kappa}_i)}-\sum_{i>m} \frac{F^{ii}h_{11i}^2}{\kappa_{1}^2}-\frac{2\varepsilon}{1-\varepsilon}\sum_{i>m} \frac{F^{ii}\overline{R}_{01i1}^2}{\kappa_1(\kappa_1-\tilde{\kappa}_i)}\\
&\quad 2\varepsilon \sum_{1<i \leq m} \frac{F^{ii}h_{11i}^2}{\kappa_1(\kappa_1-\tilde{\kappa}_i)}-\sum_{1<i \leq m} \frac{F^{ii}h_{11i}^2}{\kappa_{1}^2}-\frac{2\varepsilon}{1-\varepsilon}\sum_{1<i \leq m} \frac{F^{ii}\overline{R}_{01i1}^2}{\kappa_1(\kappa_1-\tilde{\kappa}_i)} \\
\end{align*} for some $0<\varepsilon<1$ to be determined later.

Since 
\[\frac{1}{\kappa_1-\tilde{\kappa}_i} \leq 1,\] we have
\[-\frac{2\varepsilon}{1-\varepsilon}\sum_{i \neq 1} \frac{F^{ii}\overline{R}_{01i1}^2}{\kappa_1(\kappa_1-\tilde{\kappa}_i)} \geq -\frac{2\varepsilon}{1-\varepsilon} \frac{C}{\kappa_1}\sum_{i=1}^{n} F^{ii}\] for some $C>0$ depending on $\overline{R}$. Next, for $1<i\leq m$, we have
\[2\varepsilon \frac{F^{11}h_{11i}^2}{\kappa_1(\kappa_1-\tilde{\kappa}_i)}-\frac{F^{ii}h_{11i}^2}{\kappa_{1}^2}=(2\varepsilon\kappa_1-1)\frac{F^{ii}h_{11i}^2}{\kappa_{1}^2} \geq 0\] for $\kappa_1$ being sufficiently large; and for $i > m$, we have
\[2\varepsilon \frac{F^{ii}h_{11i}^2}{\kappa_1(\kappa_1-\tilde{\kappa}_i)}-\frac{F^{ii}h_{11i}^2}{\kappa_{1}^2}=\frac{(2\varepsilon-1)\kappa_1+\tilde{\kappa}_i}{\kappa_1-\tilde{\kappa}_i}\frac{F^{ii}h_{11i}^2}{\kappa_{1}^2} \geq 0 \quad \text{if $\kappa_i \geq 0$}.\]

Therefore, we arrive at  
\begin{gather}\label{final inequality}
\begin{split}
0&\geq \left(\frac{Nu}{u-a}-1\right)\sum_{i}F^{ii}\kappa_{i}^2+\left(\alpha \phi' - C - \frac{2\varepsilon}{1-\varepsilon}\frac{C}{\kappa_1}\right)\sum_{i} F^{ii}+kF\kappa_1-C\alpha-C \\
&\quad -\sum_{p\neq q} \frac{F^{pp,qq}h_{pp1}h_{qq1}}{\kappa_1}-\frac{F^{11}h_{111}^2}{\kappa_{1}^2}+2\sum_{i \neq 1}\frac{F^{ii}h_{1ii}^2}{\kappa_1(\kappa_1-\tilde{\kappa}_i)}+N\frac{F^{11}u_{1}^2}{(u-a)^2}\\
&\quad +\sum_{i  \in I} \frac{(2\varepsilon-1)\kappa_1+\tilde{\kappa}_i}{\kappa_1-\tilde{\kappa}_i}\frac{F^{ii}h_{11i}^2}{\kappa_{1}^2}+N\sum_{i=1}^{n}\frac{F^{ii}u_{i}^2}{(u-a)^2}+\frac{\Psi_{11}}{\kappa_1} - \frac{N}{u-a} \sum_{k} \left(\sum_{i} F^{ii}h_{iik}\Phi_k\right),
\end{split}
\end{gather} where
\[I:=\{i>m: \text{$\kappa_i<0$ and $(2\varepsilon-1)\kappa_1+\tilde{\kappa}_i<0$}\}.\]

It is from here we shall continue the derivation based upon different conditions listed in theorem \ref{curvature estimate}.

\section{The Prescribed Curvature Equation in General} \label{C^2 estimates 2}

In this section, we prove theorem \ref{curvature estimate} for the first three cases i.e.
\begin{align}
\text{Case 1:}& \quad k=n-1, n \geq 3 \quad \text{or} \quad k=n-2, n \geq 5 \label{2k>n} \\
\text{Case 2:}& \quad \text{$M$ is semi-convex} \label{semi-convex} \\
\text{Case 3:}& \quad \text{$M$ is ($k+1$)-convex} \label{k+1 convex}
\end{align}

\begin{remark}
Recall from lemma \ref{Li-Ren-Wang} that we do not need to handle Case 3.
\end{remark}

For these cases, we shall employ the test function
\[Q=\log \kappa_{\max} - N\log u + \alpha \Phi.\] By proceeding exactly as in section \ref{C^2 estimates 1}, or simply by substituting $a=0$ in \eqref{final inequality} we obtain that
\begin{gather} \label{final inequality in general 1}
\begin{split}
0&\geq \left(N-1\right)\sum_{i}F^{ii}\kappa_{i}^2+\left(\alpha \phi' - C-\frac{2\varepsilon}{1-\varepsilon}\frac{C}{\kappa_1}\right)\sum_{i} F^{ii}+kF\kappa_1-C\alpha-C \\
&\quad -\sum_{p\neq q} \frac{F^{pp,qq}h_{pp1}h_{qq1}}{\kappa_1}-\frac{F^{11}h_{111}^2}{\kappa_{1}^2}+2\sum_{i \neq 1}\frac{F^{ii}h_{1ii}^2}{\kappa_1(\kappa_1-\tilde{\kappa}_i)}+NF^{11}\frac{u_{1}^2}{u^2}\\
&\quad +\sum_{i  \in I} \frac{(2\varepsilon-1)\kappa_1+\tilde{\kappa}_i}{\kappa_1-\tilde{\kappa}_i}\frac{F^{ii}h_{11i}^2}{\kappa_{1}^2}+N\sum_{i \in I}\frac{F^{ii}u_{i}^2}{u^2}+\frac{\Psi_{11}}{\kappa_1} - \frac{N}{u} \sum_{k} \left(\sum_{i} F^{ii}h_{iik}\Phi_k\right)
\end{split}
\end{gather}

Let us deal with the third line first. For the first term, we claim that 
\begin{claim} \label{claim 1}
Under either conditions \eqref{2k>n} or \eqref{semi-convex}, we have 
\[\sum_{i  \in I} \frac{(2\varepsilon-1)\kappa_1+\tilde{\kappa}_i}{\kappa_1-\tilde{\kappa}_i}\frac{F^{ii}h_{11i}^2}{\kappa_{1}^2} \geq 0\] by choosing appropriate values for $\varepsilon$.
\end{claim}
\begin{proof}[Proof of the Claim]
It suffices to prove the numerator of the quotient 
\[\frac{\kappa_1+\tilde{\kappa}_i}{\kappa_1-\tilde{\kappa}_i}\] is positive when $\kappa_i<0$.

If $k$ and $n$ satisfy condition \eqref{2k>n}, then by lemma \ref{negative kappa} we have that 
\[(2\varepsilon-1)\kappa_1+\tilde{\kappa}_i=\kappa_1+\kappa_i-1>(2\varepsilon-1)\kappa_1-\frac{n-k}{k}\kappa_1-1=\left(2\varepsilon-\frac{n}{k}\right)\kappa_1-1>0\] by choosing
\[\frac{n}{2k} < \varepsilon(n,k) < 1;\] note that since $2k>n$, the inequality is valid.

If $M$ is semi-convex i.e. $\kappa_i \geq -\eta$, then
\[(2\varepsilon-1)\kappa_1+\tilde{\kappa}_i \geq (2\varepsilon-1)\kappa_1-\eta-1 \geq 0\] by choosing any
\[\frac{1}{2}<\varepsilon<1.\]
\end{proof}

With this claim and by throwing away excessive positive terms, we have that 
\begin{gather} \label{final inequality in general 2}
\begin{split}
0&\geq \left(N-1\right)\sum_{i}F^{ii}\kappa_{i}^2+(\alpha \phi' - C)\sum_{i} F^{ii}-C\alpha-C \\
&\quad -\sum_{p\neq q} \frac{F^{pp,qq}h_{pp1}h_{qq1}}{\kappa_1}-\frac{F^{11}h_{111}^2}{\kappa_{1}^2}+2\sum_{i \neq 1}\frac{F^{ii}h_{1ii}^2}{\kappa_1(\kappa_1-\tilde{\kappa}_i)}\\
&\quad +\frac{\Psi_{11}}{\kappa_1} - \frac{N}{u} \sum_{k} \left(\sum_{i} F^{ii} h_{iik}\Phi_k\right)
\end{split}
\end{gather}

Next, we estimate the remaining two terms in the third line. 
\begin{claim} \label{remaining two terms}
\[\frac{\Psi_{11}}{\kappa_1} - \frac{N}{u} \sum_{k} \left(\sum_{i} F^{ii} h_{iik}\Phi_k\right) \geq -C\kappa_1-C\alpha-CN-C\] for some $C>0$ depending on $\overline{R}, \norm{\Psi}_{C^2}, \norm{\Phi}_{C^1}$ and $\norm{\Sigma}_{C^1}$.
\end{claim}
\begin{proof}[Proof of Claim \ref{remaining two terms}]
The term $\Phi_{11}$ can have the following rough estimate:
\begin{align*}
\Phi_{11}&=d_{V}^2\Psi(\nabla_{E_1}V,\nabla_{E_1}V)+d_{V}\Psi(\nabla_{E_1,E_1}^{2}V)\\
&\quad + 2d_{V}d_{\nu} \Psi(\nabla_{E_1}V,\nabla_{E_1}V)+d_{\nu}^2\Psi(\nabla_{E_1}\nu,\nabla_{E_1}\nu) \\
&\quad + d_{\nu}\Psi(\nabla_{E_1,E_1}^2\nu) \\
&\geq \sum_{k} h_{1k1}(d_{\nu}\Psi)(E_k)-Ch_{11}^2-Ch_{11}-C \\
&=\sum_{k} (h_{11k}+\overline{R}_{01k1})(d_{\nu}\Psi)(E_k)-Ch_{11}^2-Ch_{11}-C \\
&\geq \sum_{k} h_{11k}(d_{\nu}\Psi)(E_k)-C\kappa_{1}^2-C\kappa_1-C
\end{align*} for some $C>0$ depending on $\overline{R}$ and $\norm{\Psi}_{C^2}$. Note that
\begin{align*}
&\ \sum_{k} \frac{h_{11k}}{\kappa_1}(d_{\nu}\Psi)(E_k)- \frac{N}{u} \sum_{k} \left(\sum_{i} F^{ii}h_{iik}\Phi_k\right) \\
=&\ \sum_{k} \frac{h_{11k}}{\kappa_1}(d_{\nu}\Psi)(E_k) - N \sum_{k} F_k\frac{\Phi_k}{u} \\
=&\ \sum_{k} \left(N\frac{u_k}{u}-\alpha\Phi_k\right)(d_{\nu}\Psi)(E_k)-N\sum_k \left[h_{kk}(d_{\nu}\Psi)(E_k)+(d_{V}\Psi)(\nabla_{E_k}V)\right]\frac{\Phi_k}{u}\\
=&\ N\sum_{k}\frac{h_{kk}\Phi_k}{u}(d_{\nu}\Psi)(E_k)-N\sum_{k} h_{kk}(d_{\nu}\Psi)(E_k)\frac{\Phi_k}{u} \\
&\ -\alpha\sum_{k}d_{\nu}(\Psi)(E_k)-N\sum_{k} (d_{V}\Psi)(\nabla_{E_k}V)\frac{\Phi_k}{u}\\
\geq &\ -C\alpha-CN,
\end{align*} where the second equality is due to \eqref{1st critical 1} and \eqref{differentiate once}, and the third equality is due to lemma \ref{geometric formulas}.

The result follows.
\end{proof}

With this claim, it follows that
\begin{gather} \label{final inequality in general 3}
\begin{split}
0&\geq (N-1)\sum_{i}F^{ii}\kappa_{i}^2+(\alpha \phi' - C)\sum_{i} F^{ii}-C(\kappa_1+\alpha+N+1) \\
&\quad -\sum_{p\neq q} \frac{F^{pp,qq}h_{pp1}h_{qq1}}{\kappa_1}-\frac{F^{11}h_{111}^2}{\kappa_{1}^2}+2\sum_{i \neq 1}\frac{F^{ii}h_{1ii}^2}{\kappa_1(\kappa_1-\tilde{\kappa}_i)}\\
\end{split}
\end{gather}

Now we can proceed to eliminate the second line by applying the concavity inequalities due to Ren-Wang and Lu i.e. lemma \ref{Ren-Wang} and lemma \ref{Lu}.

\subsection{When $k=n-1, n \geq 3$ or $k=n-2, n \geq 5$}

\indent

In this case, by \eqref{Codazzi} and \eqref{HMW}, we have for some $0<\varepsilon<1$ that
\begin{align*}
& -\sum_{p\neq q} \frac{F^{pp,qq}h_{pp1}h_{qq1}}{\kappa_1}-\frac{F^{11}h_{111}^2}{\kappa_{1}^2}+2\sum_{i \neq 1}\frac{F^{ii}h_{1ii}^2}{\kappa_1(\kappa_1-\tilde{\kappa}_i)}\\
=& [\varepsilon+(1-\varepsilon)]\left(-\sum_{p\neq q} \frac{F^{pp,qq}h_{pp1}h_{qq1}}{\kappa_1}-\frac{F^{11}h_{111}^2}{\kappa_{1}^2}\right)+2\sum_{i \neq 1}\frac{F^{ii}(h_{ii1}+\overline{R}_{0i1i})^2}{\kappa_1(\kappa_1-\tilde{\kappa}_i)}\\
= & \varepsilon \left(-\sum_{p\neq q} \frac{F^{pp,qq}h_{pp1}h_{qq1}}{\kappa_1}-\frac{F^{11}h_{111}^2}{\kappa_{1}^2}\right)+2\sum_{i \neq 1}\frac{F^{ii}(h_{ii1}+\overline{R}_{0i1i})^2}{\kappa_1(\kappa_1-\tilde{\kappa}_i)}\\
& -(1-\varepsilon)\sum_{p \neq q} \frac{F^{pp,qq}h_{pp1}h_{qq1}}{\kappa_1}-(1-\varepsilon)\frac{F^{11}h_{111}^2}{\kappa_{1}^2}\\
\geq &-\frac{\beta \varepsilon }{\kappa_1}\left(\sum_{i=1}^{n} F^{ii}h_{ii1}\right)^2-2 \varepsilon \sum_{i \neq 1} \frac{F^{ii}h_{ii1}^2}{\kappa_1(\kappa_1-\tilde{\kappa}_i)} + 2\varepsilon\sum_{i \neq 1} \frac{F^{ii}h_{ii1}^2}{\kappa_1(\kappa_1-\tilde{\kappa}_i)}-\frac{2\varepsilon}{1-\varepsilon}\sum_{i \neq 1} \frac{F^{ii}\overline{R}_{0i1i}^2}{\kappa_1(\kappa_1-\tilde{\kappa}_i)} \\
&-(1-\varepsilon)\sum_{p \neq q} \frac{F^{pp,qq}h_{pp1}h_{qq1}}{\kappa_1}-(1-\varepsilon)\frac{F^{11}h_{111}^2}{\kappa_{1}^2},\\
\end{align*} where in the last inequality we have applied lemma \ref{Ren-Wang} to $\xi_i=h_{ii1}$ and divide by $\kappa_{1}^2$. Cancelling out the sums for $F^{ii}h_{ii1}^2$ and invoking the concavity of our operator $F=\sigma_k$ with \eqref{differentiate once} i.e. 
\[-\sum F^{pp,qq}h_{pp1}h_{qq1} \geq -\frac{k-1}{k}\frac{\left(\sum_{i=1}^{n} F^{ii}h_{ii1}\right)^2}{F}\] and 
\[\sum_{i=1}^{n} F^{ii}h_{ii1}=\phi'd_{V}\Psi(E_1)+h_{11}\Psi_{\nu}(E_1),\]
we obtain that
\begin{align*}
& -\sum_{p\neq q} \frac{F^{pp,qq}h_{pp1}h_{qq1}}{\kappa_1}-\frac{F^{11}h_{111}^2}{\kappa_{1}^2}+2\sum_{i \neq 1}\frac{F^{ii}h_{1ii}^2}{\kappa_1(\kappa_1-\tilde{\kappa}_i)}\\
=& -\frac{\beta \varepsilon }{\kappa_1}\left(\sum_{i=1}^{n} F^{ii}h_{ii1}\right)^2-(1-\varepsilon)\sum_{p \neq q} \frac{F^{pp,qq}h_{pp1}h_{qq1}}{\kappa_1} \\
&-\frac{2\varepsilon}{1-\varepsilon}\frac{C}{\kappa_1}\sum_{i=1}^{n} F^{ii} -(1-\varepsilon)\frac{F^{11}h_{111}^2}{\kappa_{1}^2} \\
\geq & -C\kappa_1-\frac{2\varepsilon}{1-\varepsilon}\frac{C}{\kappa_1}\sum_{i=1}^{n} F^{ii} -(1-\varepsilon)\frac{F^{11}h_{111}^2}{\kappa_{1}^2}.
\end{align*}

Thus, \eqref{final inequality in general 3} becomes:
\[0\geq \left(N-1\right)\sum_{i}F^{ii}\kappa_{i}^2+(\alpha \phi' - C)\sum_{i} F^{ii}-(1-\varepsilon)\frac{F^{11}h_{111}^2}{\kappa_{1}^2}-C(\kappa_1+\alpha+N+1).\]
For the remaining third order term, we use the first order critical equation \eqref{1st critical 1} to get that
\[(1-\varepsilon)\frac{F^{11}h_{111}^2}{\kappa_{1}^2} \leq (1-\varepsilon)F^{11} \cdot C(N^2\kappa_{1}^2+\alpha^2).\] Finally, choosing $\varepsilon>0$ very close to $1$ so that the inequality becomes
\[0 \geq (N-C)\sum F^{ii}\kappa_{i}^2 + (\alpha \phi' - C) \sum F^{ii} - C\kappa_1-C.\] Moreover, since on $\Sigma$, $\phi'$ is bounded from below by a positive controlled constant and so by choosing $\alpha$ large enough, we may assume that the coefficient $\alpha\phi'-C \geq 0$ and hence the second term is non-negative. By lemma \ref{kappa squared times f_i}, the inequality reduces to 
\[0 \geq (N-C)\frac{k}{n}\kappa_1\sigma_k - C\kappa_1 - C.\] Choosing $N$ large enough yields the estimate.

\begin{remark}
One should be able to realize that the proof of theorem \ref{theorem 2k>n} relies largely on the validity of Ren-Wang's concavity inequality. In particular, if that inequality holds for other values of $k$, then through an exact same proof, the theorem would hold for those values of $k$ as well. In fact, Ren-Wang has conjectured that their concavity inequality would hold for all $k$ satisfying $2k>n$. Therefore, if one can perform some tedious computation and prove that the inequality indeed holds for $2k>n$, then the theorem would hold for $k$-convex solutions with $2k>n$.

In our earliest draft, we proved the theorem for $2k>n$ by assuming Ren-Wang's conjecture is true. Later, we changed it to the current form i.e. only for $k=n-2$ and $k=n-1$, because only these two cases have been verified by Ren-Wang \cite{Ren-Wang-1, Ren-Wang-2}.
\end{remark}
\subsection{When $\Sigma$ is semi-convex or ($k+1$)-convex} \label{semi-convex section}

\indent

For cases 2 and 3, this time we do not have a concavity inequality as handy as that of Ren-Wang to handle the second line of \eqref{final inequality in general 3}. However, we could overcome this difficulty by employing Lu's concavity inequality \eqref{Lu's inequality} which is weaker but more generic, along with an iteration argument which is inspired by Guan-Ren-Wang \cite{Guan-Ren-Wang} and Yang \cite{Yang}.

\begin{remark}
Recall that according to lemma \ref{Li-Ren-Wang}, it is sufficient to prove the semi-convex case only.
\end{remark}

Let $\varepsilon>0$ be some small number which will be chosen later and let $\delta_0=\frac{1}{2}$. Pick an arbitrary $\delta_1 \in (0,1)$, say $\delta_1=\frac{1}{3}$, then trivially we have $\kappa_1 \geq \delta_1 \kappa_1$. Now by lemma \ref{Lu}, there exists some $\delta_2>0$ and for this $\delta_2$, if $\kappa_2 \leq \delta_{2} \kappa_1$, then Lu's inequality \eqref{Lu's inequality} holds for $l=1$. 

We will soon deal with the situation in which Lu's inequality holds; here the key argument is that if the reverse happens i.e. $\kappa_2>\delta_2 \kappa_1$, then we continue to pick some $\delta_3>0$ by lemma \ref{Lu} and for this $\delta_3>0$, we do the same analysis as above: either the process stops here and we have Lu's inequality to be analyzed, or we have $\kappa_3>\delta_3 \kappa_1$ and the process goes on.

Suppose that the process goes on until $\kappa_k>\delta_k \kappa_1$ i.e. we have $\kappa_i>\delta_i\kappa_1$ for $1 \leq i \leq k$. Then we immediately have the following by assuming $\kappa_1$ is sufficiently large:
\begin{align*}
\sigma_k(\kappa)&\geq \kappa_1\cdots \kappa_k - C_{n,k} \cdot \kappa_1\cdots \kappa_{k-1} \cdot \eta \\
&=\kappa_1\cdots\kappa_{k-1}(\kappa_k-C\eta) \\
&\geq \kappa_1 \cdot (\delta_2\kappa_1)\cdots (\delta_{k-1}\kappa_1)\cdot (\delta_k \kappa_1-C\eta) \\
&\geq C\delta_2\cdots \delta_k \kappa_{1}^{k} \quad \text{if $\Sigma$ is semi-convex i.e. $\kappa_i \geq -\eta$.}
\end{align*} which implies the desired bound for $\kappa_1$.

Now suppose that the process stops at some index $m \leq l < k$ and then Lu's inequality \eqref{Lu's inequality} holds for $l$:
\[-\sum_{p \neq q} \frac{\sigma_{k}^{pp,qq}\xi_{p}\xi_{q}}{\sigma_k}+\frac{\left(\sum_{i}\sigma_{k}^{ii}\xi_i\right)^2}{\sigma_{k}^2} \geq (1-\varepsilon)\frac{\xi_{1}^2}{\kappa_{1}^2}-\frac{1}{2}\sum_{i>l}\frac{\sigma_{k}^{ii}\xi_{i}^2}{\kappa_1\sigma_{k}}.\]

Taking $F=\sigma_k$, $\xi_i=h_{ii1}$ and multiply the inequality by $\frac{\sigma_k}{\kappa_1}$, we have that 
\[-\sum_{p \neq q} \frac{F^{pp,qq}h_{pp1}h_{qq1}}{\kappa_1} \geq (1-\varepsilon)\frac{h_{111}^2}{\kappa_{1}^3}\sigma_k-\frac{1}{2}\sum_{i>l}\frac{F^{ii}h_{ii1}^2}{\kappa_{1}^2}-\frac{F_{1}^2}{\kappa_1\sigma_k}.\]

The second line of \eqref{final inequality in general 3} can now be estimated as follows:
\begin{gather} \label{the second line by Lu}
\begin{split}
&-\sum_{p\neq q} \frac{F^{pp,qq}h_{pp1}h_{qq1}}{\kappa_1}-\frac{F^{11}h_{111}^2}{\kappa_{1}^2}+2\sum_{i \neq 1}\frac{F^{ii}h_{1ii}^2}{\kappa_1(\kappa_1-\tilde{\kappa}_i)}\\
\geq & - \frac{F_{1}^2}{\kappa_1\sigma_k}+\frac{h_{111}^2}{\kappa_{1}^2}\left(\frac{1-\varepsilon}{\kappa_1}\sigma_k-\sigma_{k}^{11}\right)+\sum_{i>l}\left[\frac{3/2}{\kappa_1(\kappa_1-\tilde{\kappa}_i)}-\frac{1}{2\kappa_{1}^2}\right]F^{ii}h_{ii1}^2-\frac{C_{\overline{R}}}{\kappa_1}\sum_{i=1}^{n} F^{ii}\\
\geq & -C\kappa_1+(1-\varepsilon)\sigma_{k}(\kappa|1)\frac{h_{111}^2}{\kappa_{1}^3}-\varepsilon\frac{F^{11}h_{111}^2}{\kappa_{1}^2}-\frac{C_{\overline{R}}}{\kappa_1}\sum_{i=1}^{n} F^{ii}
\end{split}
\end{gather}
where we have applied \eqref{Codazzi} and \eqref{HMW}, and 
used
\[\sigma_k(\kappa)=\kappa_1\sigma_{k-1}(\kappa|1)+\sigma_k(\kappa|1)\] and 
\[\frac{3/2}{\kappa_1(\kappa_1-\tilde{\kappa}_i)}-\frac{1}{2\kappa_{1}^2}=\frac{2\kappa_1+\tilde{\kappa}_i}{2\kappa_{1}^2(\kappa_1-\tilde{\kappa}_i)} \geq 0 \quad \text{by semi-convexity}.\]

Substituting \eqref{the second line by Lu} into \eqref{final inequality in general 3}, we obtain that
\begin{gather} \label{final inequality for semi-convex}
\begin{split}
0&\geq (N-1)\sum_{i} \kappa_{i}^2 F^{ii} -\varepsilon \frac{F^{11}h_{111}^2}{\kappa_{1}^2} \\
&\quad +(\alpha\phi'- C)\sum_{i} F^{ii} + (1-\varepsilon)\sigma_{k}(\kappa|1)\frac{h_{111}^2}{\kappa_{1}^3} \\
&\quad -C\kappa_1-C\alpha-CN-C.
\end{split}
\end{gather} We proceed to show how to deal with the terms involving $\varepsilon$.
By the critical equation \eqref{1st critical 1}, we have that 
\begin{align*}
-\varepsilon \frac{F^{11}h_{111}^2}{\kappa_{1}^2}&=-\varepsilon F^{11}\left(N\frac{u_1}{u}-\alpha \Phi_1\right)^2 \\
&\geq - C\varepsilon N^2F^{11}\kappa_{1}^2-C \varepsilon \alpha^2 F^{11} \\
&\geq - C\varepsilon N^2 F^{11}\kappa_{1}^2
\end{align*} and 
\begin{align*}
(1-\varepsilon)\sigma_{k}(\kappa|1)\frac{h_{111}^2}{\kappa_{1}^3}&\geq -C\kappa_2 \cdots \kappa_k \cdot \eta \cdot \frac{1}{\kappa_1}\cdot \left(N\frac{u_1}{u}-\alpha\Phi_1\right)^2 \\
&\geq -CN^2\kappa_1\cdots\kappa_k-C\alpha^2\frac{\kappa_2\cdots\kappa_k}{\kappa_1} \\
&\geq -CN^2\kappa_1\cdots\kappa_k  \quad \text{if $\kappa_i \geq -\eta$.}
\end{align*} 

Recall that on $\Sigma$, $\phi'$ is bounded from below by a positive controlled constant and so by choosing $\alpha$ large enough, we may assume that the coefficient $\alpha\phi'- C$ is positive. Therefore, by lemma \ref{sigma_k formulas}, lemma \ref{sigma_l} and choosing $\alpha$ large enough, we have that
\begin{align*}
&\ (\alpha\phi'-C)\sum_{i} F^{ii} + (1-\varepsilon)\sigma_{k}(\kappa|1)\frac{h_{111}^2}{\kappa_{1}^3} \\
=&\ (\alpha\phi'-C)(n-k+1)\sigma_{k-1}+(1-\varepsilon)\sigma_{k}(\kappa|1)\frac{h_{111}^2}{\kappa_{1}^3}\\
\geq &\ C\alpha\kappa_1\cdots\kappa_{k-1}-CN^2\kappa_1\cdots\kappa_k \\
=&\ (\kappa_1\cdots\kappa_{k-1})\cdot (C\alpha-CN^2\kappa_k)
\end{align*}

If $C\alpha-CN^2\kappa_k\geq 0$, then we are done dealing with this term. Otherwise, we would have $\kappa_k \geq CN$ by choosing $\alpha=N^3$, which will then imply that 
\begin{align*}
\sigma_k(\kappa)&\geq \kappa_1\cdots\kappa_k-C\kappa_1\cdots\kappa_{k-1}\cdot \eta \\
&= \kappa_1\cdots \kappa_{k-1}(\kappa_k-C\eta) \\
&\geq \kappa_1\cdots \kappa_{k-1}(CN-C\eta)\\
&\geq C\kappa_1
\end{align*} by choosing $N$ sufficiently large. Therefore, we may assume that $C\alpha-CN^2\kappa_k\geq 0$ and remove the second line of \eqref{final inequality for semi-convex} to have that 
\begin{align*}
0&\geq (N-1)\sum_{i} \kappa_{i}^2 F^{ii} -\varepsilon \frac{F^{11}h_{111}^2}{\kappa_{1}^2}-C\kappa_1-C\alpha-CN-C. \\
&\geq (N-1)\sum_{i} \kappa_{i}^2 F^{ii} - C\varepsilon N^2 F^{11}\kappa_{1}^2-C\kappa_1-C\alpha-CN-C\\
&\geq (N-C)\sum_i \kappa_{i}^2 F^{ii}-C\kappa_1-C\alpha-CN-C\\
&\geq CN\kappa_1\sigma_k-C\kappa_1-C\alpha-CN-C\\
&\geq (CN-C)\kappa_1-C
\end{align*} where we have chosen $\varepsilon=\frac{1}{N^2}$ and applied lemma \ref{kappa squared times f_i}. The desired estimate for $\kappa_1$ follows by choosing $N$ large enough.

\section{The Prescribed Curvature Measure Type Equation} \label{C^2 estimates 3}
In this section, we prove theorem \ref{prescribed curvature measure} by deriving $C^0, C^1$ and $C^2$ estimates for the prescribed curvature measure type equation
\[\sigma_k(\kappa[\Sigma])=\langle X,\nu\rangle^p \psi(X).\]
\subsection{The $C^0$ Estimate}
\begin{lemma}
Let (i) $1 \leq k \leq n-1, p \in (-\infty,0)\cup (0,1]$, or (ii) $k=n, p \in (-\infty,0)\cup (0,1)$.
If $\Sigma=\{(r(z),z): z \in M\}$ is a closed $k$-convex hypersurface in $\overline{M}$ which is strictly star-shaped with respect to the origin and satisfies \eqref{curvature measure equation} for some positive $\psi \in C^2(\Sigma)$, then there exist $C_1,C_2>0$ depending only on $n,k, \min_{M} \psi, \max_{M} \psi$ such that
\[C_1 \leq \min r \leq \max r \leq C_2.\]

For $k=n$ and $p=1$, the same result holds under the additional assumption that $\min_{M}\psi>1$.
\end{lemma}
\begin{proof}
Our proof follows that of Guan-Lin-Ma \cite[Lemma 2.2]{Guan-Lin-Ma} and Yang \cite[Theorem 3.2]{Yang}. Observe that the equation can be written as 
\[\sigma_{k}(h_{j}^{i})=\langle X,\nu\rangle^p \overline{\psi}(X)\] where
\[\langle X , \nu \rangle = \frac{\phi^2}{\sqrt{\phi^2+|\nabla'r|^2}},\quad \overline{\psi}=\phi^{-(n+1)}\psi;\] see (2.12) in \cite{Guan-Lin-Ma}.
Since $\Sigma$ is compact, $r$ attains its maximum at some point $z_0$. Since at this point, $\nabla' r = 0$, it follows that
\begin{align*}
h_{j}^{i}&=\frac{1}{\phi^2\sqrt{\phi^2+|\nabla' r|^2}}\left(\delta_{ik}-\frac{r^ir^k}{\phi^2+|\nabla'r|^2}\right)\left(-\phi r_{kj} + 2\phi'r_kr_j+\phi^2\phi'\delta_{kj}\right)\\
&=\frac{1}{\phi^3}\left(-\phi r_{ij}+\phi^2\phi'\delta_{ij}\right) \\
&\geq \frac{\phi'}{\phi}\delta_{ij}
\end{align*} where we have used the fact that $r_{ij} \leq 0$ at $z_0$. Hence, by ellipticity of the equation operator,
\[\psi=\phi^{n+1-p} \cdot \sigma_k(h_{j}^{i}) \geq \phi^{n+1-p} \cdot \sigma_{k}\left(\frac{\phi'}{\phi}\delta_{ij}\right) = \phi^{n+1-p}\cdot \binom{n}{k}\left(\frac{\phi'}{\phi}\right)^{k}=\binom{n}{k}(\phi')^{k}\phi^{n+1-p-k}.\]

Similarly, at the point where $r$ attains its minimum, we have that
\[\psi \leq \binom{n}{k}(\phi')^{k}\phi^{n+1-p-k}.\]

This proof works fine unless $p=1$ and $k=n$, for which we would need the additional condition that $\min \psi>1$.

\end{proof}
\subsection{The $C^1$ Estimate} \label{C^1 estimate}

\begin{theorem}
If $\Sigma=\{(r(z),z): z \in M\}$ is a closed $k$-convex hypersurface in $\overline{M}$ which is strictly star-shaped with respect to the origin and satisfies \eqref{curvature measure equation} for some positive $\psi \in C^2(\Sigma)$, then there exist $C>0$ depending only on $n,k,p, \min_{M} \psi, |\psi|_{C^1}$ and the curvature $\overline{R}$ of $\overline{M}$ such that
\[\max_{M} |\nabla' r| \leq C.\]
\end{theorem}
\begin{remark}
A remarkable feature of this result is that it holds for all $p \neq 0$ without any barrier conditions. See \cite[Lemma 3.1]{Chen-Li-Wang}, \cite[Lemma 5.1]{Guan-Ren-Wang} and \cite[Lemma 3.2]{Spruck-Xiao} for the $C^1$ estimate of the general curvature equation under barrier conditions.
\end{remark}
\begin{proof}
Our proof employs the test function of Chen-Li-Wang \cite[Lemma 3.1]{Chen-Li-Wang} but follows the derivation of Huang-Xu \cite[Lemma 2.3]{Huang-Xu}.
We first observe that 
\[\langle V,\nu\rangle = \frac{\phi(r)^2}{\sqrt{\phi(r)^2+|\nabla' r|^2}}\] and so
this gradient bound is equivalent to $\langle V,\nu \rangle \geq C>0$, due to the $C^0$ estimate.

Consider the test function
\[P=\gamma \left(\Phi\right)-\log u\] where $\gamma$ is to be determined later. Suppose that $P$ attains its maximum at $z_0 \in \Sigma$. If $V$ is parallel to the normal direction $\nu$ at $z_0$, then $\langle V,\nu \rangle=|V|$ and we are done. Assume this is not the case, then we can choose a local orthonormal frame $\{E_1,\ldots,E_n\}$ such that $\langle V,E_1\rangle \neq 0$ and $\langle V, E_i \rangle=0$ for $i \geq 2$. Note also that $V=\langle V, E_1 \rangle E_1 + \langle V,\nu\rangle \nu$.

Then at $z_0$, we have that
\begin{align}
0&=P_i=\gamma'\Phi_i-\frac{u_i}{u} \label{C1 1st critical}\\
0&\geq P_{ii}=\gamma''\Phi_{i}^2+\gamma'\Phi_{ii}-\frac{u_{ii}}{u}+\frac{u_{i}^2}{u^2} \nonumber \\
&=\gamma''\Phi_{i}^2+\gamma'\Phi_{ii}-\frac{u_{ii}}{u}+(\gamma')^2\Phi_{i}^2 \nonumber \\
&=\gamma'\Phi_{ii}-\frac{u_{ii}}{u}+[\gamma''+(\gamma')^2]\Phi_{i}^2\label{C1 2nd critical}
\end{align}
By \eqref{C1 1st critical}, we have that 
\begin{equation}
h_{11}=u\gamma' \quad \text{and} \quad h_{i1}=0 \quad \forall\ i \geq 2. \label{h11}
\end{equation} We can then rotate the coordinate system such that $\{E_1,\ldots,E_n\}$ are the principal curvature directions of the second fundamental form so that $h_{ij}=\kappa_i \delta_{ij}$.

Contracting \eqref{C1 2nd critical} with $F^{ii}$ and apply lemma \ref{geometric formulas}, we have that
\begin{align*}
0 & \geq \gamma'\sum F^{ii} \Phi_{ii}-\frac{1}{u} \sum F^{ii}u_{ii}+[\gamma''+(\gamma')^2]\sum F^{ii}\Phi_{i}^2 \\
&=\gamma'\sum F^{ii}(\phi'-u\kappa_i)-\frac{1}{u}\sum F^{ii} \left[\sum_k(h_{iik}-\overline{R}_{0iik})\Phi_k+\phi'\kappa_i-\kappa_{i}^2u\right]\\
&\quad +[\gamma''+(\gamma')^2]\sum F^{ii}\Phi_{i}^2 \\
&\geq \gamma'\phi'\sum F^{ii} - u \gamma' \cdot k \Psi - \frac{1}{u}\sum F^{ii}h_{ii1}\Phi_1+\frac{1}{u}\sum F^{ii}\overline{R}_{0ii1}\Phi_1-\frac{\phi'}{u}\cdot k \Psi+\sum F^{ii}\kappa_{i}^2 \\
&\quad + [\gamma''+(\gamma')^2]F^{11}\Phi_{1}^2
\end{align*}

Differentiating \eqref{prescribed curvature measure}, we have that
\[F^{ii}h_{ii1}=pu^{p-1}u_1\cdot \psi + u^p \cdot \psi_1.\] Substituting this and $\Psi=u^p\psi$ into the above inequality, we obtain that
\begin{align*}
0&\geq \gamma' \phi'\sum F^{ii} - k\gamma' u^{p+1}\psi-\frac{\Phi_1}{u}u^p\left(p\psi\frac{u_1}{u}+\psi_1\right)\\
&\quad -\frac{\Phi_1}{u}\sum F^{ii}\overline{R}_{0i1i}-k\phi'u^{p-1}\psi+\sum F^{ii}\kappa_{i}^2 + [\gamma''+(\gamma')^2]F^{11}\Phi_{1}^2
\end{align*} where we have also used the skew symmetry property $\overline{R}_{0ii1}=-\overline{R}_{0i1i}$.

By \eqref{C1 1st critical}, we have that 
\[\frac{u_1}{u}=\gamma'\Phi_1.\] The inequality then becomes
\begin{align*}
0&\geq \gamma' \phi'\sum F^{ii} - k\gamma' u^{p+1}\psi-pu^{p-1}\gamma'\psi\Phi_{1}^2-\Phi_1u^{p-1}\psi_1\\
&\quad -\frac{\Phi_1}{u}\sum F^{ii}\overline{R}_{0i1i}-k\phi'u^{p-1}\psi+\sum F^{ii}\kappa_{i}^2 + [\gamma''+(\gamma')^2]F^{11}\Phi_{1}^2 \\
\end{align*}
Now, using
\[|V|^2=\langle V,E_1\rangle^2+\langle V,\nu\rangle^2 \quad \text{i.e.} \quad \Phi_{1}^2=|V|^2-u^2,\] we have
\begin{gather}\label{C1 inequality 1}
\begin{split}
0&\geq [\gamma''+(\gamma')^2][\ |V|^2-u^2]F^{11}+\gamma'\phi'\sum F^{ii} - k \gamma'\psi u^{p+1}-[\ |V|^2-u^2]p\gamma'\psi u^{p-1} \\
&\quad -\Phi_1\psi_1 u^{p-1}-k\phi'\psi u^{p-1}+\sum F^{ii}\kappa_{i}^2-\frac{\Phi_1}{u}\sum F^{ii}\overline{R}_{0i1i}.
\end{split}
\end{gather}
Rearranging the terms in \eqref{C1 inequality 1}, we have that
\begin{gather}\label{C1 inequality 2}
\begin{split}
&\ [\gamma''+(\gamma')^2]F^{11}u^2+(k-p)\gamma'\psi u^{p+1} \\
&\ +u^{p-1}[\ |V|^2p\gamma'\psi + \Phi_1\psi_1 + k \phi'\psi] \\
\geq &\ [\gamma''+(\gamma')^2]\cdot |V|^2 F^{11} + \gamma'\phi'\sum F^{ii}\\
&\ + \sum F^{ii}\kappa_{i}^2-\frac{\Phi_1}{u}\sum F^{ii}\overline{R}_{0i1i}
\end{split}
\end{gather}
Now, choose 
\[\gamma(s)=\frac{\alpha p}{s}\quad \text{i.e.} \quad \text{$\gamma'(s)=-\frac{\alpha p}{s^2}$ and $\gamma''(s)=\frac{2\alpha p}{s^3}$},\] where $\alpha>0$ is some possibly large constant depending on the given value of $p$. We claim that
\begin{claim}
$(k-p)\gamma'\psi u^{p+1}+u^{p-1}[\ |V|^2p\gamma'\psi + \Phi_1\psi_1 + k \phi'\psi] \leq 0$
\end{claim}
\begin{proof}
If $p<0$ or $p>t$, then 
\begin{align*}
&\ (k-p)\gamma'\psi u^{p+1}+p|V|^2\gamma'\psi u^{p-1}\\
=&\ \gamma'\psi u^{p-1}[(k-p)u^2+p|V|^2] \\
=&\ -\alpha p \frac{\psi}{\Phi^2}u^{p-1}[(k-p)u^2+p|V|^2]\\
\leq &\ -\frac{\alpha p^2 \psi}{2\Phi^2} |V|^2u^{p-1},
\end{align*} where we have assumed that
\[u^2 \leq \frac{p}{2(p-k)}|V|^2.\] The desired inequality follows by choosing $\alpha$ large enough.

If $0<p \leq k$, then we immediately have that
\[(k-p)\gamma'\psi u^{p+1} \leq 0\] and 
\[|V|^2p\gamma'\psi + \Phi_1\psi_1 + k \phi'\psi=-\frac{\alpha p^2}{\Phi^2}|V|^2\psi+\Phi_1\psi_1+k\phi'\psi \leq 0\] by choosing $\alpha$ large enough.
\end{proof}
\begin{claim}
$\sum F^{ii}\kappa_{i}^2-\frac{\Phi_1}{u}\sum F^{ii}\overline{R}_{0i1i} \geq -C \sum F^{ii}$.
\end{claim}
\begin{proof}
Note that in \cite{Chen-Li-Wang}, they applied \eqref{h11} to have that $\sum F^{ii}\kappa_{i}^2 \geq F^{11}\kappa_{1}^2 = u^2(\gamma')^2 F^{11}$, but this would cancel out the important term in $[\gamma''+(\gamma')^2]F^{11}u^2$. So here we just throw it away by positivity i.e. $\sum F^{ii}\kappa_{i}^2 \geq 0$.

For the second term, since $V=\langle V,E_1\rangle E_1+\langle V,\nu\rangle\nu$, it follows that $V \perp \Span\{E_2,\ldots,E_n\}$. Since we also have $E_1,\nu \perp \Span \{E_2,\ldots,E_n\}$, we can choose a coordinate system such that $\overline{e}_1 \perp \Span\{E_2,\ldots,E_n\}$, which implies that the pair $\{V,\overline{e}_1\}$ and $\{\nu, E_1\}$ would lie in the same plane, and 
\[\Span\{E_2,\ldots,E_n\}=\Span\{\overline{e}_2,\ldots,\overline{e}_n\}.\]

Thus, by choosing $E_i=\overline{e}_i$ for $2 \leq i \leq n$, we may write
\begin{align*}
\nu&=\langle \nu, \overline{e}_0\rangle \overline{e}_0 + \langle \nu, \overline{e}_1\rangle \overline{e}_1=\frac{u}{\phi}\overline{e}_0+\langle \nu,\overline{e}_1\rangle \overline{e}_1,\\
E_1&=\langle E_1,\overline{e}_0\rangle \overline{e}_0+\langle E_1,\overline{e}_1\rangle \overline{e}_1.
\end{align*} The curvature tensor can then be computed as 
\begin{align*}
\overline{R}_{0i1i}&=\overline{R}(\nu, E_i, E_1, E_i) \\
&=\frac{u}{\phi}\langle E_1,\overline{e}_0\rangle \overline{R}(\overline{e}_0, \overline{e}_i, \overline{e}_0, \overline{e}_i)+\langle \nu, \overline{e}_1\rangle \langle E_1,\overline{e}_1\rangle\overline{R}(\overline{e}_1, \overline{e}_i, \overline{e}_1, \overline{e}_i)\\
&=\frac{u}{\phi}\langle E_1,\overline{e}_0\rangle \overline{R}(\overline{e}_0, \overline{e}_i, \overline{e}_0, \overline{e}_i)-\frac{u\langle \nu, \overline{e}_1\rangle^2}{\langle E_1,V\rangle}\overline{R}(\overline{e}_1, \overline{e}_i, \overline{e}_1, \overline{e}_i)\\
&=u\left(\frac{\langle E_1,\overline{e}_0\rangle}{\phi}\overline{R}(\overline{e}_0, \overline{e}_i, \overline{e}_0, \overline{e}_i)-\frac{\langle \nu, \overline{e}_1\rangle^2}{\langle E_1,V\rangle}\overline{R}(\overline{e}_1, \overline{e}_i, \overline{e}_1, \overline{e}_i)\right),
\end{align*} where the second equality follows from $0=\overline{R}_{ijk0}$ \cite[Lemma 2.1]{Chen-Li-Wang} and the third equality is due to $0=\langle V,\overline{e}_1\rangle$. This cancels out the $\frac{1}{u}$ factor in the front and the result follows.
\end{proof}
With these claims, the inequality \eqref{C1 inequality 2} is reduced to 
\begin{equation}
[\gamma''+(\gamma')^2]F^{11}u^2 \geq [\gamma''+(\gamma')^2]\cdot |V|^2 F^{11} + (\gamma'\phi'+C)\sum F^{ii}. \label{C1 inequality 3}
\end{equation}

Finally, we can finish the derivation by one more claim. 

\begin{claim}
$[\gamma''+(\gamma')^2]\cdot |V|^2 F^{11} + (\gamma'\phi'+C)\sum F^{ii} \geq C F^{11}$.
\end{claim}
\begin{proof}
By our choice of $\gamma$, we have that 
\begin{align*}
&\ [\gamma''+(\gamma')^2]\cdot |V|^2 F^{11} + \gamma'\phi'\sum F^{ii} \\
=&\ \left[\frac{2\alpha p}{\Phi^2}+\frac{\alpha^2p^2}{\Phi^4}\right]\cdot |V|^2 F^{11}-\alpha p \frac{\phi'}{\Phi^2}\sum F^{ii}.
\end{align*}

If $p<0$, the first term is still positive by choosing $\alpha$ large enough and the second term will immediately imply that 
\[-\alpha p \frac{\phi'}{\Phi^2}\sum F^{ii} \geq C F^{11}.\]

If $p \geq 0$, then by \eqref{h11} we have that $h_{11}=u\gamma' \leq 0$. Hence, we have that 
\begin{align*}
F^{11}=\sigma_{k-1}(\kappa|1)=\sigma_{k-1}(\kappa)-\kappa_1\sigma_{k-2}(\kappa|1) \geq \sigma_{k-1}(\kappa) \geq c(n,k)\sum F^{ii}
\end{align*} by lemma \ref{sigma_k formulas}. The result then follows by choosing $\alpha$ large enough.
\end{proof}

With this claim, the inequality \eqref{C1 1st critical} yields
\[[\gamma''+(\gamma')^2]F^{11}u^2 \geq C F^{11}\] and the bound $u \geq C$ is obtained by choosing $\alpha$ large enough so that $\gamma''+(\gamma')^2 > 0$.
\end{proof}

\subsection{The $C^2$ Estimate}
\indent

Finally, we derive the $C^2$ estimate for the prescribed curvature measure type equation i.e. 
\[\sigma_{k}(\kappa(X))=\langle X,\nu\rangle^p\psi(X) \quad \text{where $p \in (-\infty,0)\cup (0,1]$}.\]

\begin{proof}[Proof of Theorem \ref{theorem curvature measure}]
For this case, we will use the test function of Spruck-Xiao \cite{Spruck-Xiao}
\[Q=\log \kappa_{\max} - \log (u-a) + \alpha \Phi\] where $u:=\langle X,\nu\rangle$ and $u \geq 2a$, and $\alpha>0$ is a constant to be chosen later.

By proceeding exactly as in section \ref{C^2 estimates 1}, or simply by substituting $N=1$ in \eqref{final inequality}, we obtain that

\begin{gather}\label{final inequality p}
\begin{split}
0&\geq \left(\frac{u}{u-a}-1\right)\sum_{i}F^{ii}\kappa_{i}^2+(\alpha \phi' - C)\sum_{i} F^{ii}+kF\kappa_1-C\alpha-C \\
&\quad -\sum_{p\neq q} \frac{F^{pp,qq}h_{pp1}h_{qq1}}{\kappa_1}-\frac{F^{11}h_{111}^2}{\kappa_{1}^2}+2\sum_{i \neq 1}\frac{F^{ii}h_{1ii}^2}{\kappa_1(\kappa_1-\tilde{\kappa}_i)}+\frac{F^{11}u_{1}^2}{(u-a)^2}\\
&\quad +\sum_{i  \in I} \frac{(2\varepsilon-1)\kappa_1+\tilde{\kappa}_i}{\kappa_1-\tilde{\kappa}_i}\frac{F^{ii}h_{11i}^2}{\kappa_{1}^2}+\sum_{i \in I}\frac{F^{ii}u_{i}^2}{(u-a)^2}+\frac{\Psi_{11}}{\kappa_1} - \frac{1}{u-a} \sum_{k} \left(\sum_{i} F^{ii}h_{iik}\Phi_k\right).
\end{split}
\end{gather}
Recall that the index set $I$ is defined as 
\[I=\{i>m: \text{$\kappa_i<0$ and $\kappa_1+\widetilde{\kappa}_i<0$}\}.\]
There are two major obstacles here: (1) unlike in claim \ref{claim 1}, this time it is not clear if the quotient 
\[\frac{(2\varepsilon-1)\kappa_1+\tilde{\kappa}_i}{\kappa_1-\tilde{\kappa}_i}\] is positive without extra conditions; and (2) it is difficult to handle the large negative term 
\[-\frac{F^{11}h_{111}^2}{\kappa_{1}^2};\] we have to make full use of all the positive terms in \eqref{final inequality p} combined with the special structure of the right-hand side $\Psi=u^p\psi$.

The first trick is to split the $\sum_i \kappa_{i}^2 F^{ii}$ term into
\begin{gather}\label{splitting}
\begin{split}
\frac{a}{u-a}\sum_{i}\kappa_{i}^2 F^{ii}&=\left(\frac{n-1}{n}\frac{a}{u-a}+\frac{1}{n}\frac{a}{u-a}\right)\sum_{i=1}^{n}F^{ii}\kappa_{i}^2. \\
&\geq \frac{n-1}{n}\frac{a}{u-a}\sum_{i=1}^{n} F^{ii}\kappa_{i}^2 + \frac{1}{n}\frac{a}{u-a}F^{11}\kappa_{1}^2 + \frac{1}{n}\frac{a}{u-a}\sum_{i \in I} F^{ii}\kappa_{i}^2
\end{split}
\end{gather}

We claim that 
\begin{claim} \label{claim 2}
\[\frac{1}{n}\frac{a}{u-a}\sum_{i \in I}F^{ii}\kappa_{i}^2+\sum_{i  \in I} \frac{(2\varepsilon-1)\kappa_1+\tilde{\kappa}_i}{\kappa_1-\tilde{\kappa}_i}\frac{F^{ii}h_{11i}^2}{\kappa_{1}^2}+ \sum_{i \in I}\frac{F^{ii}u_{i}^2}{(u-a)^2} \geq 0\] and 
\[\frac{1}{n}\frac{a}{u-a}F^{11}\kappa_{1}^2-\frac{F^{11}h_{111}^2}{\kappa_{1}^2}+\frac{F^{11}u_{1}^2}{(u-a)^2} \geq 0,\] where $I=\{i>m: \kappa_1+\widetilde{\kappa}_i<0\}$.
\end{claim}

\begin{remark}
The proof for this claim is in fact independent of this $\frac{1}{n}$ coefficient. The reason for the splitting \eqref{splitting} is that even after we have used parts of the sum to prove claim \ref{claim 2}, we are still left with a complete sum $\frac{n-1}{n}\frac{a}{u-a}\sum F^{ii}\kappa_i$ or simply the term $F^{11}\kappa_{1}^2$, which are both of the same order as $\kappa_1$ and will crucially be used in the subsequent derivation. That is, the most important is not the magnitude of its coefficient but merely its presence.
\end{remark}

\begin{proof}[proof of the claim]
Observe the trivial fact that 
\[\frac{(2\varepsilon-1)\kappa_1+\tilde{\kappa}_i}{\kappa_1-\tilde{\kappa}_i} \geq -1\] and recall from the first order critical equation \eqref{1st critical 1} that
\[\frac{h_{11i}^2}{\kappa_{1}^2}=\left(\frac{u_i}{u-a}-\alpha \Phi_i\right)^2.\]

Also, note that for $i \in I$, we have $|\kappa_i| \geq \kappa_1$. It is then straightforward to have that
\begin{align*}
&\frac{1}{n}\frac{a}{u-a}\sum_{i \in I}\kappa_{i}^2f_i+\sum_{i  \in I} \frac{(2\varepsilon-1)\kappa_1+\tilde{\kappa}_i}{\kappa_1-\tilde{\kappa}_i}\frac{F^{ii}h_{11i}^2}{\kappa_{1}^2}+ \sum_{i \in I}\frac{F^{ii}u_{i}^2}{(u-a)^2} \\
\geq & \sum_{i \in I} \frac{1}{n} \frac{a}{u-a}F^{ii}\kappa_{i}^2-\frac{F^{ii}h_{11i}^2}{\kappa_{1}^2}+\frac{F^{ii}u_{i}^2}{(u-a)^2} \\
= & \sum_{i \in I} \frac{1}{n} \frac{a}{u-a}F^{ii}\kappa_{i}^2-F^{ii}\left(\frac{u_i}{u-a}-\alpha \Phi_i\right)^2+\frac{F^{ii}u_{i}^2}{(u-a)^2} \\
= & \sum_{i \in I} \frac{1}{n} \frac{a}{u-a}F^{ii}\kappa_{i}^2+2\alpha F^{ii} \frac{u_i}{u-a}\Phi_i-\alpha^2F^{ii}\Phi_{i}^2 \\
\geq & \sum_{i \in I} \frac{1}{n} \frac{a}{u-a}F^{ii}|\kappa_{i}|^2-C\alpha F^{ii}|\kappa_i|-C\alpha^2 F^{ii} \\
\geq & 0 \quad \text{by assuming $\kappa_1$ is sufficiently large.}
\end{align*} 

The proof for the second inequality is very similar:
\begin{align*}
&\ \frac{1}{n} \frac{a}{u-a}F^{11}\kappa_{1}^2-\frac{F^{11}h_{111}^2}{\kappa_{1}^2}+\frac{F^{11}u_{1}^2}{(u-a)^2} \\
\geq &\ \frac{1}{n} \frac{a}{u-a}F^{11}\kappa_{1}^2-C\alpha F^{11}\kappa_1-C\alpha^2 F^{11} \\
\geq &\ 0 \quad \text{by assuming $\kappa_1$ is sufficiently large.}
\end{align*}
\end{proof}

With this claim, the inequality \eqref{final inequality p} implies that
\begin{gather} \label{p inequality 1}
\begin{split}
0 & \geq \frac{n-1}{n}\frac{a}{u-a}\sum_{i=1}^{n} F^{ii}\kappa_{i}^2 + (\alpha \phi'- C)\sum_{i=1}^{n} F^{ii} \\
&\quad - \sum_{p \neq q} \frac{F^{pp,qq}h_{pp1}h_{qq1}}{\kappa_1}+2\sum_{i \neq 1} \frac{F^{ii}h_{ii1}^2}{\kappa_1(\kappa_1-\tilde{\kappa}_i)} \\
&\quad +\frac{\Psi_{11}}{\kappa_1} - \frac{1}{u-a} \sum_{k} \left(\sum_{i} F^{ii}h_{iik}\Phi_k\right)+kF\kappa_1-C\alpha-C.
\end{split}
\end{gather}

We can proceed to invoke the special structure of $\Psi$ to handle the first two terms in the third line of \eqref{p inequality 1}. We first find, by lemma \ref{geometric formulas}, the first order equation \eqref{1st critical 1} and the expression $F=\Psi=u^{p}\psi$, that
\begin{align*}
\Psi_{11}&=(u^p\psi)_{11}\\
&=\psi\left[pu^{p-1}u_{11}+p(p-1)u^{p-2}u_{1}^2\right]+2\psi_1\cdot pu^{p-1}u_1+u^p\cdot \psi_{11}\\
&=\psi\cdot pu^{p-1}\left[\sum_{k}(h_{11k}-\overline{R}_{011k})\Phi_k+\phi'\kappa_1-\kappa_{1}^2u\right]\\
&\quad +\psi \cdot p(p-1)u^{p-2}u_{1}^2+2p\psi_1u^p\frac{u_{1}}{u}+u^p\psi_{11}\\
&\geq \psi\cdot pu^{p-1}\sum_k h_{11k}\Phi_k-\psi\cdot pu^{p}\kappa_{1}^2+\psi\cdot p(p-1)u^{p-2}u_{1}^2-C\kappa_1-C\\
&\geq \frac{pF}{u}\sum_k h_{11k}\Phi_k-pF\kappa_{1}^2+p(p-1)F\frac{u_{1}^2}{u^2}-C\kappa_1-C.
\end{align*} for some $C$ depending on $p, \overline{R}, \norm{\Sigma}_{C^1}$ and $\norm{\psi}_{C^2}$.
Then, by noting that
\begin{align*}
&\ - \frac{1}{u-a} \sum_{k} \left(\sum_{i} F^{ii}h_{iik}\Phi_k\right) + \frac{pF}{u}\sum_k \frac{h_{11k}}{\kappa_1}\Phi_k\\
=&\ -\sum_k \frac{F_k\Phi_k}{u-a}+\frac{pF}{u}\sum_k \frac{h_{11k}}{\kappa_1}\Phi_k \\
=&\ -\sum_{k} \left(\psi \cdot pu^{p-1}u_k+u^p\cdot \psi_k\right)\frac{\Phi_k}{u-a}+\frac{pF}{u}\sum_{k}\left(\frac{u_k}{u-a}-\alpha \Phi_k\right)\Phi_k \\
= &\ -\frac{pF}{u}\sum_{k} \frac{u_k\Phi_k}{u-a}+\frac{pF}{u}\sum_{k}\frac{u_k\Phi_k}{u-a}-u^{p}\sum_{k}\frac{\psi_k\Phi_k}{u-a}-\alpha \frac{pF}{u}\sum_{k} \Phi_{k}^2 \\
\geq &-C\alpha-C,
\end{align*} it follows from \eqref{p inequality 1} that 

\begin{gather} \label{p inequality 2}
\begin{split}
0 & \geq \frac{n-1}{n}\frac{a}{u-a}\sum_{i=1}^{n} F^{ii}\kappa_{i}^2 + (\alpha \phi'- C)\sum_{i=1}^{n} F^{ii} \\
&\quad - \sum_{p \neq q} \frac{F^{pp,qq}h_{pp1}h_{qq1}}{\kappa_1}+2\sum_{i \neq 1} \frac{F^{ii}h_{1ii}^2}{\kappa_1(\kappa_1-\tilde{\kappa}_i)}\\
&\quad +(k-p)F\kappa_1+p(p-1)\frac{F}{\kappa_1}\frac{u_{1}^2}{u^2}-C\alpha-C.
\end{split}
\end{gather}

It remains to handle the second line and the term 
\[p(p-1)\frac{F}{\kappa_1}\frac{u_{1}^2}{u^2},\] which can be negative if $p \in (0,1)$.

\subsubsection{Case A: $\kappa_{2}\cdots\kappa_{k-1} \geq \alpha^{-\frac{k-2}{n}}$}

\indent

By concavity of the operator $\sigma_{k}^{1/k}$ i.e.
\[\frac{1}{k}\sigma_{k}^{\frac{1}{k}-1}\sum_{p \neq q} \sigma_{k}^{pp,qq}h_{pp1}h_{qq1}+\frac{1}{k}\left(\frac{1}{k}-1\right)\sigma_{k}^{\frac{1}{k}-2}\sum_{p,q}\sigma_{k}^{pp}h_{pp1}\sigma_{k}^{qq}h_{qq1} \leq 0,\] we are able to obtain that 
\[-\sum_{p\neq q} F^{pp,qq}h_{pp1}h_{qq1} \geq -\frac{k-1}{k}\frac{F_{1}^2}{F} \geq -C\kappa_{1}^2.\]

For the sum $\sum F^{ii}$, we use lemma \ref{sigma_k formulas} and lemma \ref{sigma_l} to have that
\[\sum_{i=1}^{n} F^{ii}=(n-k+1)\sigma_{k-1}(\kappa)\geq c(n,k)\kappa_1\cdots\kappa_{k-1}.\]

Hence,
\begin{align*}
&\ - \sum_{p \neq q} \frac{F^{pp,qq}h_{pp1}h_{qq1}}{\kappa_1}+ (\alpha \phi'- C)\sum_{i=1}^{n} F^{ii}+p(p-1)\frac{F}{\kappa_1}\frac{u_{1}^2}{u^2} \\
\geq &\ -C\kappa_1 + (\alpha\phi' - C)c(n,k)\kappa_1\cdots\kappa_{k-1}+p(p-1)\frac{F}{\kappa_1}\frac{u_{1}^2}{u^2}\\
\geq &\ -C\kappa_1+C\alpha^{\frac{n-k+2}{n}}\kappa_1 \\
\geq &\ 0 \quad \text{by choosing $\alpha$ sufficiently large,}
\end{align*} where the estimate $p(p-1)\frac{F}{\kappa_1}\frac{u_{1}^2}{u^2}\geq -C\kappa_1$ is given in remark \ref{upper bound for p}. Now, the inequality \eqref{p inequality 2} along with lemma \ref{kappa squared times f_i}
yields
\begin{align*}
0 &\geq  \frac{n-1}{n}\frac{a}{u-a}\sum_{i=1}^{n} F^{ii}\kappa_{i}^2+(k-p)F\kappa_1-C\alpha-C \\
  &\geq CF\kappa_1-C\alpha-C
\end{align*} as desired.

\subsubsection{Case B: $\kappa_{2}\cdots\kappa_{k-1} \leq \alpha^{-\frac{k-2}{n}}$}

\indent

For this case, we do not have a large enough lower bound for $\sum F^{ii}$ and this is where we have to restrict the range of $p$. We will proceed by the same iteration argument as in section \ref{semi-convex section} and apply Lu's inequality \eqref{Lu's inequality}. Let $\delta=1/3$, $\delta_0=\min\{k/n,1/2\}$ and let $\varepsilon>0$ to be chosen later. We consider two sub-cases according to lemma \ref{Lu}.

\textbf{Case B.1.} For each $2 \leq l \leq k$, there exists some $\delta_l>0$ such that $\kappa_{l} \geq \delta_l \kappa_1$.

\indent

Then we immediately have that 
\[\alpha^{-\frac{k-2}{n}} \geq \kappa_2\cdots \kappa_{k-1} \geq \delta_2 \cdots \delta_{k-1}\kappa_{1}^{k-1}\] and the estimate for $\kappa_1$ follows.

\textbf{Case B.2.} There exists $m \leq l<k$ such that \eqref{Lu's inequality} holds.

\indent

That is, we have
\[-\sum_{p \neq q} \frac{\sigma_{k}^{pp,qq}\xi_{p}\xi_{q}}{\sigma_k}+\frac{\left(\sum_{i}\sigma_{k}^{ii}\xi_i\right)^2}{\sigma_{k}^2} \geq (1-\varepsilon)\frac{\xi_{1}^2}{\kappa_{1}^2}-\delta_0\sum_{i>l}\frac{\sigma_{k}^{ii}\xi_{i}^2}{\kappa_1\sigma_{k}}.\]

Substituting $\xi_{i}=h_{ii1}$, the second line in \eqref{p inequality 2} then becomes
\begin{align*}
&\ - \sum_{p \neq q} \frac{F^{pp,qq}h_{pp1}h_{qq1}}{\kappa_1}+2\sum_{i \neq 1} \frac{F^{ii}h_{1ii}^2}{\kappa_1(\kappa_1-\tilde{\kappa}_i)} \\
\geq &\ (1-\varepsilon)\sigma_k\frac{h_{111}^2}{\kappa_{1}^3}+\sum_{i>l} \left[\frac{1}{\kappa_1(\kappa_1-\tilde{\kappa}_i)}-\frac{\delta_0}{\kappa_{1}^2}\right]F^{ii}h_{ii1}^2-\frac{F_{1}^2}{\kappa_1\sigma_{k}}-\frac{C_{\overline{R}}}{\kappa_1}\sum F^{ii}.
\end{align*}
By lemma \ref{negative kappa} and our choice of $\delta_0$, we have that 
\[\frac{1}{\kappa_1(\kappa_1-\tilde{\kappa}_i)}-\frac{\delta_0}{\kappa_{1}^2} \geq 0.\]

Therefore, by lemma \ref{kappa squared times f_i}, the first order equation \eqref{1st critical 1} and the expression $F=\Psi=u^p\psi$, we have that
\begin{align*}
&\ \frac{n-1}{n}\frac{a}{u-a}\sum_{i=1}^{n} F^{ii}\kappa_{i}^2+C\alpha\sum F^{ii}- \sum_{p \neq q} \frac{F^{pp,qq}h_{pp1}h_{qq1}}{\kappa_1}+2\sum_{i \neq 1} \frac{F^{ii}h_{1ii}^2}{\kappa_1(\kappa_1-\tilde{\kappa}_i)}+p(p-1)\frac{F}{\kappa_1}\frac{u_{1}^2}{u^2} \\
\geq &\ \frac{n-1}{n}\frac{a}{u-a}\sum_{i=1}^{n} F^{ii}\kappa_{i}^2+(1-\varepsilon)\sigma_k\frac{h_{111}^2}{\kappa_{1}^3}-\frac{F_{1}^2}{\kappa_1\sigma_{k}}+p(p-1)\frac{F}{\kappa_1}\frac{u_{1}^2}{u^2}\\
= &\ CF\kappa_1 + (1-\varepsilon)\frac{F}{\kappa_1} \left(\frac{u_{1}}{u-a}-\alpha\Phi_1\right)^2 - \frac{(\psi\cdot pu^{p-1}u_1+u^p\psi_1)^2}{\kappa_1\sigma_k}+p(p-1)\frac{F}{\kappa_1}\frac{u_{1}^2}{u^2} \\
= &\ CF\kappa_1+ (1-\varepsilon)\frac{F}{\kappa_1}\left(\frac{u_{1}^2}{(u-a)^2}-2\alpha\frac{u_1\Phi_1}{u-a}+\alpha^2\Phi_{1}^2\right) \\
&\ - \frac{1}{\kappa_1\sigma_k}\left(\psi^2p^2u^{2p-2}u_{1}^2+2\psi p u^{2p-1}u_1\psi_1+u^{2p}\psi_{1}^2\right)+p(p-1)\frac{F}{\kappa_1}\frac{u_{1}^2}{u^2}\\
\geq &\ CF\kappa_1+\left[(1-\varepsilon)\frac{F}{\kappa_1}\frac{u_{1}^2}{(u-a)^2}-C\alpha\right]-p^2\frac{F}{\kappa_1}\frac{u_{1}^2}{u^2}-C+p(p-1)\frac{F}{\kappa_1}\frac{u_{1}^2}{u^2}\\
=&\ CF\kappa_1 + \left[(1-\varepsilon)\frac{F}{\kappa_1}\frac{u_{1}^2}{(u-a)^2}-p\frac{F}{\kappa_1}\frac{u_{1}^2}{u^2}\right]+\left[-p^2\frac{F}{\kappa_1}\frac{u_{1}^2}{u^2}+p^2\frac{F}{\kappa_1}\frac{u_{1}^2}{u^2}\right]-C\alpha-C\\
=&\ CF \kappa_1 -\varepsilon \frac{F}{\kappa_1}\frac{u_{1}^2}{(u-a)^2} + \left[\frac{F}{\kappa_1}\frac{u_{1}^2}{(u-a)^2}-p\frac{F}{\kappa_1}\frac{u_{1}^2}{u^2}\right]-C\alpha-C\\
\geq &\ CF\kappa_1-C\varepsilon F\kappa_1 + (1-p)\frac{F}{\kappa_1}\frac{u_{1}^2}{u^2}- C\alpha - C \\
\end{align*}

With this estimate, \eqref{p inequality 2} then implies that 
\begin{equation}
0 \geq (C-C\varepsilon)F\kappa_1+(1-p)\frac{F}{\kappa_1}\frac{u_{1}^2}{u^2}+(k-p) F\kappa_1 - C\alpha - C. \label{p inequality 3}
\end{equation}

If $p \leq 1$, then the proof is complete by choosing a sufficiently small $\varepsilon>0$.

\end{proof}
\begin{remark}\label{upper bound for p}
It can be clearly seen that a sufficient condition for the estimate to hold is
\begin{equation}
(1-p)\frac{F}{\kappa_1}\frac{u_{1}^2}{u^2}+(k-p) F\kappa_1 \geq 0. \label{p k inequality}
\end{equation}

Note that
\[u_{1}^2=\kappa_{1}^2\Phi_{1}^2=\kappa_{1}^2(|V|^2-u^2),\] so it follows that
\[(1-p)\frac{F}{\kappa_1}\frac{u_{1}^2}{u^2}\geq (1-p)\left(\frac{1}{\varepsilon_0}-1\right)F\kappa_1,\] where 
\[\varepsilon_0:=\inf_{\Sigma} \frac{u^2}{|V|^2}.\]

Hence, the inequality \eqref{p k inequality} holds if 
\[(1-p)\left(\frac{1}{\varepsilon_0}-1\right)+(k-p) \geq 0,\] which is true if and only if
\[p \leq 1+\varepsilon_0(k-1).\] This shows that our estimate, in fact, holds for a wider range $p \leq 1+C$ for some $C>0$ depending on $k$ and $\norm{\Sigma}_{C^1}$. Taking $k=2$ will recover Chen's upper bound for $p$ in \cite[(4.32)-(4.33)]{Chen}.

However, since this constant is not universal, we shall not state it explicitly in our theorem.

\end{remark}

\bigskip

\textbf{Funding.} No funding was received to assist with the preparation of this manuscript.

\bigskip

\textbf{Data availability.} Data sharing not applicable to this article as no datasets were generated or analyzed during
the current study.
\bibliography{refs}
\end{document}